\documentclass[11pt]{amsart}
\usepackage{amsmath, amssymb, graphicx, url, amsthm, multicol, latexsym, enumerate, bbm, mathtools, physics, microtype, mathptmx, cite, float, booktabs, comment,array,tabularx,array}
\usepackage[dvipsnames, table]{xcolor}
\usepackage{wrapfig}
\usepackage[hmargin=1in, vmargin=1in]{geometry} 
\usepackage{tikz}
\usetikzlibrary{shapes.geometric, arrows,backgrounds,patterns}

\usepackage{hyperref}

\hypersetup{
    colorlinks=true,
    linkcolor=blue,
    citecolor=magenta,
    filecolor=magenta,      
    urlcolor=magenta,
}

\newtheorem{theorem}{Theorem}[section]
\newtheorem{proposition}[theorem]{Proposition}
\newtheorem{conjecture}[theorem]{Conjecture}
\newtheorem{lemma}[theorem]{Lemma}
\newtheorem{corollary}[theorem]{Corollary}
\newtheorem{algorithm}[theorem]{Algorithm}
\theoremstyle{definition}
\newtheorem{remark}[theorem]{Remark}
\newtheorem{definition}[theorem]{Definition}

\newtheorem{notation}[theorem]{Notation}

\numberwithin{equation}{section}

\newcommand{\R}{\mathbb{R}}

\newcommand{\Z}{\mathbb{Z}}

\newcommand{\dstirling}[2]{\genfrac{[}{]}{0pt}{0}{#1}{#2}}
\newcommand{\circleat}[3]{\filldraw[#3](#1,#2) circle (2pt)}

\DeclareMathOperator{\genfkthree}{F}
\DeclareMathOperator{\genffourk}{H}
\DeclareMathOperator{\genffourkm}{G}
\DeclareMathOperator{\genfplain}{F}

\newcommand{\createTable}[2]{
    \def\n{#1}
    \begin{tikzpicture}[scale =0.7]
        \foreach \x in {0,...,\n} {
            \foreach \y in {0,...,\n} {
                \node[draw, minimum size=0.7cm] at (\x, \y) {};
            }
        }
        \foreach \m/\x/\y in {#2} {
            \node at (\x, \y) {\m};
        }
    \end{tikzpicture}
}

\DeclareMathOperator{\area}{area}
\DeclareMathOperator{\bounce}{bounce}

\usepackage{etoolbox}
\AtBeginEnvironment{proof}{\setlength{\parindent}{0.4cm}}
\AtBeginEnvironment{itemize}{\setlength{\parindent}{0.2cm}}
\AtBeginEnvironment{enumerate}{\setlength{\parindent}{0.2cm}}

\makeatletter 
\newtheorem*{rep@theorem}{\rep@title}\newcommand{\newreptheorem}[2]{%
\newenvironment{rep#1}[1]{%
\def\rep@title{\bf #2 \ref{##1}}%
\begin{rep@theorem}}%
{\end{rep@theorem}}}
\makeatother
\newreptheorem{theorem}{Theorem}

\newcommand{\ourdef}[1]{\emph{#1}}

\begin{document}


\title{Polyhedral geometry of refined $q,t$-Catalan numbers}

\author{Matthias Beck, Mitsuki Hanada, Max Hlavacek, John Lentfer, \\ Andr\'es R. Vindas-Mel\'endez, Katie Waddle }
\address{\scriptsize{Department of Mathematics, San Francisco State University\\
\url{https://matthbeck.github.io/}}}
\email{\scriptsize{mattbeck@sfsu.edu}}

\address{\scriptsize{Department of Mathematics, University of California, Berkeley\\
\url{https://math.berkeley.edu/~mhanada}}}
\email{\scriptsize{mhanada@berkeley.edu}}

\address{\scriptsize{Department of Mathematics \& Statistics, Pomona College\\
\url{https://math.berkeley.edu/~mhlava}}}
\email{\scriptsize{magda.hlavacek@pomona.edu}}

\address{\scriptsize{Department of Mathematics, University of California, Berkeley\\
\url{https://math.berkeley.edu/~jlentfer}}}
\email{\scriptsize{jlentfer@berkeley.edu}}

\address{\scriptsize{Department of Mathematics, Harvey Mudd College}, \url{https://math.hmc.edu/arvm}}
\email{\scriptsize{arvm@hmc.edu}}

\address{\scriptsize{Department of Mathematics, University of Michigan\\
\url{https://sites.google.com/view/katie-waddle}}}
\email{\scriptsize{waddle@umich.edu}}

\date{30 July 2024}


\begin{abstract}
We study a refinement of the $q,t$-Catalan numbers introduced by Xin and Zhang (2022, 2023) using tools from polyhedral geometry. 
These refined $q,t$-Catalan numbers depend on a vector of parameters $\vec{k}$ and the classical $q,t$-Catalan numbers are recovered when $\vec{k} = (1,\ldots,1)$.
We interpret Xin and Zhang's generating functions by developing polyhedral cones arising from constraints on $\vec{k}$-Dyck paths and their associated area and bounce statistics.
Through this polyhedral approach, we recover Xin and Zhang's theorem on
$q,t$-symmetry of the refined $q,t$-Catalan numbers in the cases where $\vec{k}
= (k_1,k_2,k_3)$ and $(k,k,k,k)$, give some extensions, including the case $\vec{k} = (k,k+m,k+m,k+m)$, and discuss
relationships to other generalizations of the $q,t$-Catalan numbers.
\end{abstract}

\maketitle


\section{Introduction}\label{sec:Introduction}

The $q,t$-Catalan numbers are a two-parameter deformation of the well-known Catalan numbers $\frac{1}{n+1}\binom{2n}{n}$. 
They were first introduced by Haiman in \cite{Haiman} as the bigraded Hilbert series of the sign representation of a certain $S_n$-module called the diagonal coinvariants.
An elegant way to write down the $q,t$-Catalan numbers, due to Garsia and Haglund \cite{GarsiaHaglund}, is
\begin{equation}\label{eq:qtcatalanintro}
    C_n(q,t) = \sum_{D \in \mathcal{D}_n} q^{\area(D)} \, t^{\bounce(D)},
\end{equation}
where $\mathcal{D}_n$ is the set of Dyck paths of length $n$ and area and
bounce are certain statistics on Dyck paths; we give details below in Section~\ref{sec:Background}.
We can immediately observe the symmetry of $q$ and $t$ from Haiman's original definition of $C_n(q,t)$.
A famous open problem asks for a bijection on Dyck paths that interchanges area and bounce, which would provide a combinatorial proof of $q,t$-symmetry.
For more on the history and equivalent definitions of $q,t$-Catalan numbers, see \cite[Chapter 3]{Haglund} or \cite[Section 1.5]{Loehr}.

We study a refinement $C_{\vec{k}}(q,t)$ of the $q,t$-Catalan numbers introduced by Xin and Zhang~\cite{XinZhang}.
Here $\vec{k}$ is a vector with positive integer entries that sum to $n$, and we take the sum in~\eqref{eq:qtcatalanintro} only over those Dyck paths whose north steps are precisely the entries of $\vec{k}$.
We recover $C_n(q,t)$ as the special case where $\vec{k}$ has dimension $n$ and all of its entries are~1.
It is a natural question to ask in which cases $C_{\vec{k}}(q,t)$ is symmetric in $q$ and $t$. Xin and Zhang proved symmetry when $k$ is two- or three-dimensional, while giving a counterexample in dimension four~\cite{XinZhang, xin2022qtsymmetry}.
They also proved symmetry for the case that $\vec k$ has four equal entries.
Their proofs are constructive and involved: they compute, via MacMahon's Omega operator, the rational generating function of $C_{\vec{k}}(q,t)$, from which the symmetry is plainly visible; for $\vec{k} \in \Z_{ >0 }^3$, they also give an intricate bijective proof.
One of our goals is to interpret (and re-derive) Xin--Zhang's generating functions
\begin{equation}\label{eq:intro-F}
    \genfkthree(x_1, x_2, x_3, q, t) := \sum_{ \vec k \in \Z_{ \ge 0 }^3 } C_{\vec{k}}(q,t) \, x_1^{ k_1 } x_2^{ k_2 } x_3^{ k_3 } 
\end{equation}
and
\begin{equation}\label{eq:intro-H}
  \genffourk(x, q, t) := \sum_{ k \in \Z_{ \ge 0 } } C_{(k,k,k,k)}(q,t) \, x^k
\end{equation}
through the use of polyhedral geometry: we develop polyhedral cones and certain subdivisions which appear organically from the (linear) constraints defining Dyck paths and their area and bounce statistics. 
It turns out these cones are arithmetically nice: their integer lattice-point structure are either unimodular or close to unimodular, and so their generating functions are manageable to compute.

Xin and Zhang were primarily interested in the $q,t$-symmetry of 
\[
C_{\lambda}(q,t)=\sum_{\mu(\vec{k}) = \lambda}C_{\vec{k}}(q,t),\]
where $\mu(\vec{k}) = \lambda$ means that rearranging the entries of $\vec{k}$ in decreasing order gives the partition $\lambda$.
In general, if all $C_{\vec{k}}(q,t)$ are $q,t$-symmetric for a fixed $\mu(\vec{k}) = \lambda$, then $C_{\lambda}(q,t)$ is $q,t$-symmetric.
It is also possible for $C_{\lambda}(q,t)$ to be $q,t$-symmetric even if some of the $C_{\vec{k}}(q,t)$ summands used in its definition are not; for example, 
\[C_{\lambda = (1,1,1,2)}(q,t) = C_{\vec{k} = (1,1,1,2)}(q,t) + C_{\vec{k} = (1,1,2,1)}(q,t) + C_{\vec{k} = (1,2,1,1)}(q,t) + C_{\vec{k} = (2,1,1,1)}(q,t)\] is $q,t$-symmetric, but individually, $C_{\vec{k} = (1,1,2,1)}(q,t)$ and $ C_{\vec{k} = (1,2,1,1)}(q,t)$ are not. 
It is unclear how common this symmetry phenomenon is; starting with partitions
$\lambda$ of length 4 many $C_\lambda$ are not symmetric in $q,t$.

After giving background on both $q,t$-Catalan numbers and the integer-point structure of polyhedral cones in Section~\ref{sec:Background}, we compute $\genfkthree(x_1, x_2, x_3, q, t)$ in Section~\ref{sec:k_1k_2k_3} and $\genffourk(x, q, t)$ in Section~\ref{sec:k^4}. 
Naturally, this geometric \emph{ansatz} can be used in other situations, as we exhibit
this in Section~\ref{sec:kaaa} by computing
\begin{equation}\label{eq:gfctkaaa}
  \genffourkm(x, y, q, t) := \sum_{ k, m \in \Z_{ \ge 0 } } C_{(k,k+m,k+m,k+m)}(q,t) \, x^k \, y^m ,
\end{equation}
realizing symmetry again along the way. We conjecture that $C_{(k,a,\ldots,a)}(q,t)$ is
$q,t$-symmetric for any positive $k$ and $a$.
We conclude in Section~\ref{sec:further} with various extensions of our work,
including a proof that \[ C_{\vec{k} = (k_1,\ldots,k_j,m)}(q,t) = C_{\vec{k} =
(k_1,\ldots,k_j,l)}(q,t) \, . \]
Finally, we discuss the relationship between the refined $q,t$-Catalan numbers and other generalizations of the $q,t$-Catalan numbers. 


\section{Background}\label{sec:Background}

\subsection{\texorpdfstring{$q,t$}{q,t}-Catalan numbers}

The \ourdef{Catalan numbers}~$C_n = \frac{1}{n+1} \binom{2n}{n}$ are well known to have numerous combinatorial interpretations (see, for example,~\cite{Stanley}) and  satisfy the recurrence relation 
\[C_n = \sum_{k=1}^n C_{k-1} C_{n-k} \, , \]
for~$n\geq 1$, with the initial condition~$C_0=1$.
There is not always a canonical way to create a~$q$-analogue of a combinatorial
sequence, and indeed,
the Catalan numbers have two well-known $q$-analogues. 
One of them is the \ourdef{Carlitz--Riordan~$q$-Catalan numbers}~$C_n(q)$, which are defined by the recurrence relation 
\[C_n(q) = \sum_{k=1}^n q^{k-1} \, C_{k-1}(q) \, C_{n-k}(q) \, , \]
for~$n\geq 1$, with the initial condition~$C_0(q) = 1$. 
The second is the \ourdef{MacMahon~$q$-Catalan numbers} $D_n(q)$, which are
defined via the usual $q$-integers and $q$-binomial coefficients by 
\[D_n(q) = \frac{1}{[n+1]_q}\dstirling{2n}{n}_q \, . \]

Based on computational data, Haiman first observed that the bigraded Hilbert series of the sign representation of the diagonal coinvariants was a~$q,t$-deformation of the Catalan numbers \cite{Haiman}. 
Haiman first defined the $q,t$-Catalan numbers by $C_n(q,t) := \langle e_n, \nabla e_n \rangle$, where $\nabla$ is a certain Macdonald eigenoperator for the Macdonald symmetric functions.
These~$q,t$-Catalan numbers~$C_n(q,t)$ were further studied in \cite{GarsiaHaiman}, where the first combinatorial formula was presented.
Several conjectured properties were then proven in \cite{Haiman-tq, Haiman-vanishing}.

Garsia and Haiman \cite{GarsiaHaiman} showed that their combinatorial construction of the~$q,t$-Catalan numbers~$C_n(q,t)$ included both the Carlitz-Riordan and MacMahon~$q$-Catalan numbers as specializations:
\[
C_n(q) = C_n(q,1) = C_n(1,q) \qquad  \text{ and } \qquad D_n(q) = C_n(q, \tfrac 1
q) \, q^{\binom{n}{2}} \, . \]
Garsia and Haiman's combinatorial formula for~$C_n(q,t)$ was not obviously a polynomial in~$\Z[q,t]$, but a rational function of polynomials. 
Later, an elegant characterization of the~$q,t$-Catalan numbers, which made polynomiality clear, was conjectured by Haglund \cite{HaglundBounce} and proven by Garsia and Haglund \cite{GarsiaHaglund} in terms of two statistics on Dyck paths, namely, area and bounce.

\begin{definition}
A \ourdef{Dyck path} is a lattice path on the lattice~$\Z^2$ that starts at~$(0,0)$, uses only~$(0,1)$ and~$(1,0)$ as steps (called north and east, respectively), ends at~$(n,n)$, and never goes below the line~$y=x$. 
Denote the set of all Dyck paths from~$(0,0)$ to~$(n,n)$ by~$\mathcal{D}_n$. 
For a Dyck path~$D\in \mathcal{D}_n$ we define~$\area(D)$ to be the number of complete cells between the Dyck path and the line~$y=x$.  
\end{definition}

\begin{definition}
For~$D\in \mathcal{D}_n$, the statistic~$\bounce(D)$ is due to Haglund \cite{HaglundBounce} and is defined by the following algorithm. 
Begin at~$(0,0)$ and travel north until an east step of~$D$ is encountered. 
Turn at a right angle to the east and travel east until the line~$y=x$ is reached. 
Then turn north and travel until an east step of~$D$ is encountered and repeat until~$(n,n)$ is reached. 
This defines the bounce path of~$D$. 
The bounce path ``bounces'' off the line~$y=x$ at positions~$(0,0), (j_1, j_1), (j_2, j_2), \ldots, (j_b, j_b) = (n,n)$. 
Then define 
\[ \bounce(D) := \sum_{i=1}^{b-1} n - j_i \, .\]
\end{definition}

For the purposes of this paper, we take the following as the definition of the~$q,t$-Catalan numbers~\cite{GarsiaHaglund}:
\begin{theorem}[Garsia--Haglund]
The~$q,t$-Catalan numbers are given by
    \[C_n(q,t) = \sum_{D \in \mathcal{D}_n} q^{\area(D)} \, t^{\bounce(D)}.\]
\end{theorem}

The symmetry of $C_n(q,t)$ follows immediately from Haiman's original
definition of the $q,t$-Catalan numbers. 
Thus, we obtain the following as a consequence.

\begin{corollary} For any nonnegative integer~$n$,
    \[
\sum_{D \in \mathcal{D}_n} q^{\area(D)} \, t^{\bounce(D)} = \sum_{D \in \mathcal{D}_n} q^{\bounce(D)} \, t^{\area(D)}.
\]
\end{corollary}

A notable open problem \cite[Open Problem 3.11]{Haglund} is to find a bijection on Dyck paths that interchanges area and bounce, which would prove the~$q,t$-symmetry in a combinatorial way.
Figure~\ref{fig:C_3(q,t)} provides an example of the calculation of $C_3(q,t)$, highlighting the aforementioned terminology and $q,t$-symmetry.  

\begin{figure}[ht]
    \centering
\begin{tikzpicture}[scale=0.8] 
    \foreach \x/\y in {0/1, 0/2, 1/2}{
      \fill[red!40] (\x,\y) rectangle +(1,1);
    }
    \draw[help lines] (0,0) grid (3,3); 
    \draw[dashed] (0,0) -- (3,3); 
    \coordinate (current) at (0,0);
    \draw[line width=2pt] (current) foreach \dir in {1,1,1,0,0,0}{ 
      -- ++(\dir*90:1) coordinate (current)
    };
    \coordinate (current) at (0.1,-0.1);
    \draw[blue, line width=1pt, dashed] (current) foreach \dir in {1,1,1,0,0,0}{ 
      -- ++(\dir*90:1) coordinate (current)
    };
  \end{tikzpicture}
\begin{tikzpicture}[scale=0.8] 
    \foreach \x/\y in {0/1, 1/2}{
      \fill[red!40] (\x,\y) rectangle +(1,1);
    }
    \draw[help lines] (0,0) grid (3,3); 
    \draw[dashed] (0,0) -- (3,3); 
    \coordinate (current) at (0,0);
    \draw[line width=2pt] (current) foreach \dir in {1,1,0,1,0,0}{ 
      -- ++(\dir*90:1) coordinate (current)
    };
    \coordinate (current) at (0.1,-0.1);
    \draw[blue, line width=1pt, dashed] (current) foreach \dir in {1,1,0,0,1,0}{ 
      -- ++(\dir*90:1) coordinate (current)
    };
  \end{tikzpicture}
\begin{tikzpicture}[scale=0.8] 
    \foreach \x/\y in {0/1}{
      \fill[red!40] (\x,\y) rectangle +(1,1);
    }
    \draw[help lines] (0,0) grid (3,3); 
    \draw[dashed] (0,0) -- (3,3); 
    \coordinate (current) at (0,0);
    \draw[line width=2pt] (current) foreach \dir in {1,1,0,0,1,0}{ 
      -- ++(\dir*90:1) coordinate (current)
    };
    \coordinate (current) at (0.1,-0.1);
    \draw[blue, line width=1pt, dashed] (current) foreach \dir in {1,1,0,0,1,0}{ 
      -- ++(\dir*90:1) coordinate (current)
    };
  \end{tikzpicture}
\begin{tikzpicture}[scale=0.8] 
    \foreach \x/\y in {1/2}{
      \fill[red!40] (\x,\y) rectangle +(1,1);
    }
    \draw[help lines] (0,0) grid (3,3); 
    \draw[dashed] (0,0) -- (3,3); 
    \coordinate (current) at (0,0);
    \draw[line width=2pt] (current) foreach \dir in {1,0,1,1,0,0}{ 
      -- ++(\dir*90:1) coordinate (current)
    };
    \coordinate (current) at (0.1,-0.1);
    \draw[blue, line width=1pt, dashed] (current) foreach \dir in {1,0,1,1,0,0}{ 
      -- ++(\dir*90:1) coordinate (current)
    };
  \end{tikzpicture}
\begin{tikzpicture}[scale=0.8] 
    \foreach \x/\y in {}{
      \fill[red!40] (\x,\y) rectangle +(1,1);
    }
    \draw[help lines] (0,0) grid (3,3); 
    \draw[dashed] (0,0) -- (3,3); 
    \coordinate (current) at (0,0);
    \draw[line width=2pt] (current) foreach \dir in {1,0,1,0,1,0}{ 
      -- ++(\dir*90:1) coordinate (current)
    };
    \coordinate (current) at (0.1,-0.1);
    \draw[blue, line width=1pt, dashed] (current) foreach \dir in {1,0,1,0,1,0}{ 
      -- ++(\dir*90:1) coordinate (current)
    };
  \end{tikzpicture}
    \caption{The five Dyck paths on a~$3 \times 3$ grid, used to calculate $C_3(q,t) = q^3 + q^2t + qt + qt^2 + t^3$. 
    The red boxes contribute to area. 
    The dashed blue path represents the bounce path. 
    From left to right, the area statistic is~$3,2,1,1,0$ and the bounce statistic is~$0,1,1,2,3$. }
    \label{fig:C_3(q,t)}
\end{figure}
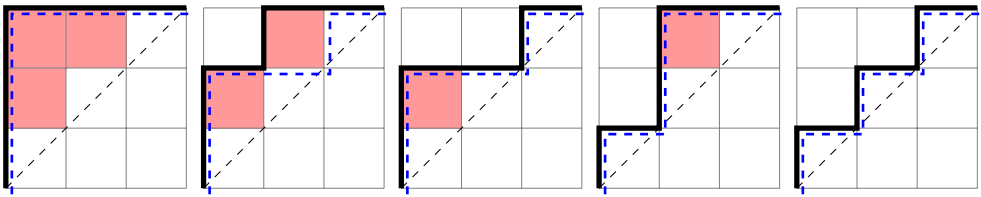

\subsection{\texorpdfstring{$\vec{k}$-Dyck paths}{k-Dyck paths}}

Xin and Zhang refined the~$q,t$-Catalan numbers by modifying Haglund's area-bounce formula \cite{XinZhang}. 
They provided three equivalent models which can be used to define a~$\vec{k}$-Dyck path. 
Here, we recall their first model.

\begin{definition}
Let~$\vec{k}=(k_1,\dots,k_m)$ be a vector of positive integers which sum to~$n$. 
A \textit{$\vec{k}$-Dyck path} is a lattice path in~$\Z^2$ from~$(0,0)$ to~$(n,n)$, which never goes below the diagonal~$y=x$, such that the north steps are of length~$k_i$, for~$1\leq i \leq m$, from bottom to top. 
The east steps are of length 1 without additional restriction. 
Denote the set of all $\vec{k}$-Dyck paths by $D_{\vec{k}}$.
Define the \textit{rank} of a point~$(x,y)$ of a~$\vec{k}$-Dyck path as~$y-x$. 
\end{definition}

\begin{remark}
In general, a~$\vec{k}$-Dyck path encodes more information than a standard Dyck path because it also includes information about how long the north steps are. 
For example, consider a~$\vec{k}$-Dyck path that contains a north step of length~$k_i$ immediately followed by a north step of length~$k_{i+1}$: if one considers this as a standard Dyck path by sending north steps of length~$k_i$ to~$k_i$ north steps of length 1, we cannot then recover the lengths~$k_i, k_{i+1}$. 
\end{remark}

We recover the~$k$-Dyck paths studied by Loehr \cite{Loehr} as a special case by using the vector~$\vec{k} = (k,k,\ldots,k)$ (and shearing the~$k$-Dyck path as necessary so that the slope of the main diagonal becomes~$1$).
We can also recover the standard Dyck paths as a special case by using the vector~$\vec{k} = (1,1,\ldots,1)$.
Xin and Zhang also extended the definition of the area and bounce of a Dyck path to~$\vec{k}$-Dyck paths (such that they are compatible with the definitions of area and bounce for $k$-Dyck paths). 
To describe area for a $\vec{k}$-Dyck path, we need an auxiliary definition~\cite[Section 2.1]{XinZhang}.

\begin{definition}\label{def:red_ranks}
Let~$D\in D_{\vec{k}}$ be a~$\vec{k}$-Dyck path. 
Define the \ourdef{red ranks} of a~$\vec{k}$-Dyck path to be the sequence~$(r_1,\ldots,r_n)$, where~$r_i$ is the rank of the starting point of the north step corresponding to~$k_i$.
\end{definition}

The red ranks can also be computed by counting complete lattice cells between the path and the line~$y=x$. 
Specifically, $r_i$ is the number of such cells in the row whose bottom left corner is the south end of the start of the~$k_i$ set of north steps.
A~$\vec{k}$-Dyck path is uniquely determined by its red ranks since they determine the location of all north steps. 
Note that the first red rank~$r_1$ is always equal to~$0$.

\begin{definition}
For~$D\in D_{\vec{k}}$, define the statistic area to be the sum of its red ranks, that is,
\[\area(D) := r_1+\cdots+r_n \, .\]
\end{definition}

We recall from \cite[Section 2.3]{XinZhang} the algorithm to compute bounce for a $\vec{k}$-Dyck path. 
For most of this paper, with the exception of Section~\ref{sec:further}, it suffices use a few special cases.

\begin{algorithm}[Xin-Zhang]\label{alg: bounce}
Input: A $\vec{k}$-Dyck path $D \in D_{\vec{k}}$.

\begin{itemize}
    \item Set $i=0$.
    \item Let $R$ be a diagram with $k_j+1$ cells in column $j$.
    \item Set $P_0 := (0,0)$ and define $R^0 := R$ to be the empty diagram.
    \item While $P_i \neq (|\vec{k}|,|\vec{k}|)$, do the following:
    \begin{enumerate}[(1)]
        \item Start a vertical path from $P_i$ towards the north, and let $v_i$ be the number of north steps traversed before reaching the starting point $Q_i = (x_i,y_i)$ of an east step;
        \item To obtain $R^{i+1}$, write $i$ in the first row of the next $v_i$ unfilled columns of $R^i$. To complete these $v_i$ columns, increment by $1$ going down;
        \item Let $h_i$ be the number of cells in $R^{i+1}$ filled by $i+1$, and set $P_{i+1} := (x_i + h_i, y_i)$;
        \item Increment $i$ to $i+1$.
    \end{enumerate}
\end{itemize}

Output: The bounce statistic $\bounce(D) := \sum_{i \geq 0} iv_i$.
\end{algorithm}

Xin and Zhang call the final tableau created during the algorithm the \textit{rank tableau} and they note that $\bounce(D)$ is equal to the sum of the entries in its first row.
For an example showing the construction of the rank tableau, see \cite[Figure 3]{XinZhang}. 
The  \textit{bounce path} is a $\vec{k}$-Dyck path, distinct from the original Dyck path $D$, where each $Q_i$ is the start of an east step of $D$. 
Note that the bounce path for $D$ does not necessarily touch the diagonal line $y=x$ (or ``bounce'' off the line) after each sequence of consecutive east steps like it does in the classical $\vec{k} = (1,\ldots,1)$ case.

For small cases, one can give a piecewise closed-form formula for bounce.
In general, this becomes more difficult as the number of entries in~$\vec{k}$ increases. 
For example, we extensively use the following characterization of bounce when~$\vec{k}$ has three parts~\cite[p. 9]{XinZhang}. 

\begin{lemma}[Xin--Zhang]
    For a~$(k_1,k_2,k_3)$-Dyck path~$D$, 
    \begin{equation}\label{eq: bounce for k_1,k_2,k_3}
\bounce(D)=
\begin{cases}
   2(k_1-r_2)+r_2+k_2-r_3-\min(r_2,k_2)  &  \text{\rm if } r_2+k_2-r_3\ge 2\min(r_2,k_2), \\
   2(k_1-r_2)+\lceil\frac{r_2+k_2-r_3}{2}\rceil  &  \text{\rm otherwise.}
\end{cases}
    \end{equation}
\end{lemma}

\noindent In Section~\ref{sec:k^4}, we also state and use an involved formula for bounce when~$\vec{k} = (k,k,k,k)$ due to Niu~\cite{Niu}.

\begin{definition}
Let~$\mathcal{D}_{\vec{k}}$ be the set of all~$\vec{k}$-Dyck paths and define the refined~$q,t$-Catalan number by
\[
C_{\vec{k}}(q,t)=\sum_{D\in\mathcal{D}_{\vec{k}}}q^{\area(D)} \, t^{\bounce(D)} .
\]
Arranging the entries of~$\vec{k}$ in decreasing order gives a partition~$\mu(\vec{k})$ and we further define
\[
C_{\lambda}(q,t)=\sum_{\mu(\vec{k})=\lambda}\sum_{D\in\mathcal{D}_{\vec{k}}}q^{\area(D)} \, t^{\bounce(D)}.
\]
\end{definition}

Xin and Zhang proved the~$q,t$-symmetry of~$C_{(k_1,k_2)}(q,t)$,~$C_{(k_1,k_2,k_3)}(q,t)$, and~$C_{(k,k,k,k)}(q,t)$ for any positive integers~$k_1,k_2,k_3$, and~$k$ \cite{XinZhang, xin2022qtsymmetry}. 
For an example, see Figure~\ref{fig:C_[2,1](q,t)}.
They also investigated the symmetry of $C_{\lambda}(q,t)$, observing that symmetry holds for partitions of length 2.
They infer symmetry in the case that~$\lambda$ is a partition of length 3 as a corollary to the~$q,t$-symmetry of~$C_{(k_1,k_2,k_3)}(q,t)$ and exhibit examples showing that symmetry does not hold in general for partitions of length 4. 
Furthermore, Xin and Zhang conjectured that symmetry of~$C_\lambda(q,t)$ holds given~$\lambda=((a+1)^s,a^{n-s})$ with~$0\leq s\leq n$.
Niu~\cite{Niu} proves this conjecture in the case that~$n=4$ using similar partition analysis methods as those employed by Xin and Zhang~\cite{xin2022qtsymmetry}.

\begin{notation}\label{notation}
    When $\vec{k}$ or $\lambda$ are given specific values in $C_{\vec{k}}(q,t)$ or $C_{\lambda}(q,t)$, respectively, it can be ambiguous which definition is being referred to. 
    In this paper, we adopt the convention that if otherwise unlabelled by $\vec{k}$ or $\lambda$, we are referring to $C_{\vec{k}}(q,t)$ when we use an expression like $C_{(1,2,3)}(q,t)$.
\end{notation}

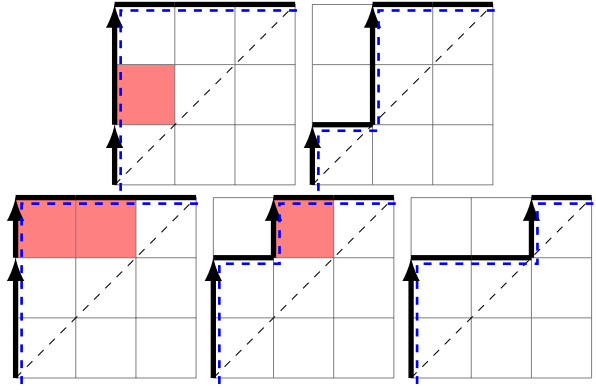
\begin{figure}
    \centering
\begin{tikzpicture}[scale=0.8] 
    \foreach \x/\y in {0/1}{
      \fill[red!50] (\x,\y) rectangle +(1,1);
    }
    \draw[help lines] (0,0) grid (3,3); 
    \draw[dashed] (0,0) -- (3,3); 
    \coordinate (current) at (0,0);
    \foreach \dir/\length in {1/1,1/2,0/1,0/1,0/1}{
      \ifnum\dir=0 
        \draw[line width=2pt] (current) -- ++(\length,0) coordinate (current);
      \else 
        \draw[line width=2pt, -latex] (current) -- ++(0,\length) coordinate (current); 
      \fi
    }
    \coordinate (currentBounce) at (0.1,-0.1); 
    \draw[blue, line width=1pt, dashed] (currentBounce) foreach \dir in {1, 1, 1, 0, 0, 0}{ 
      -- ++(\dir*90:1) coordinate (currentBounce)
    };
  \end{tikzpicture}
\begin{tikzpicture}[scale=0.8] 
    \foreach \x/\y in {}{
      \fill[red!50] (\x,\y) rectangle +(1,1);
    }
    \draw[help lines] (0,0) grid (3,3); 
    \draw[dashed] (0,0) -- (3,3); 
    \coordinate (current) at (0,0);
    \foreach \dir/\length in {1/1,0/1,1/2,0/1,0/1}{
      \ifnum\dir=0 
        \draw[line width=2pt] (current) -- ++(\length,0) coordinate (current);
      \else 
        \draw[line width=2pt, -latex] (current) -- ++(0,\length) coordinate (current); 
      \fi
    }
    \coordinate (currentBounce) at (0.1,-0.1); 
    \draw[blue, line width=1pt, dashed] (currentBounce) foreach \dir in {1, 0, 1, 1, 0, 0}{ 
      -- ++(\dir*90:1) coordinate (currentBounce)
    };
  \end{tikzpicture}
\\
\begin{tikzpicture}[scale=0.8] 
    \foreach \x/\y in {0/2,1/2}{
      \fill[red!50] (\x,\y) rectangle +(1,1);
    }
    \draw[help lines] (0,0) grid (3,3); 
    \draw[dashed] (0,0) -- (3,3); 
    \coordinate (current) at (0,0);
    \foreach \dir/\length in {1/2,1/1,0/1,0/1,0/1}{
      \ifnum\dir=0 
        \draw[line width=2pt] (current) -- ++(\length,0) coordinate (current);
      \else 
        \draw[line width=2pt, -latex] (current) -- ++(0,\length) coordinate (current); 
      \fi
    }
    \coordinate (currentBounce) at (0.1,-0.1); 
    \draw[blue, line width=1pt, dashed] (currentBounce) foreach \dir in {1, 1, 1, 0, 0, 0}{ 
      -- ++(\dir*90:1) coordinate (currentBounce)
    };
  \end{tikzpicture}
\begin{tikzpicture}[scale=0.8] 
    \foreach \x/\y in {1/2}{
      \fill[red!50] (\x,\y) rectangle +(1,1);
    }
    \draw[help lines] (0,0) grid (3,3); 
    \draw[dashed] (0,0) -- (3,3); 
    \coordinate (current) at (0,0);
    \foreach \dir/\length in {1/2,0/1,1/1,0/1,0/1}{
      \ifnum\dir=0 
        \draw[line width=2pt] (current) -- ++(\length,0) coordinate (current);
      \else 
        \draw[line width=2pt, -latex] (current) -- ++(0,\length) coordinate (current); 
      \fi
    }
    \coordinate (currentBounce) at (0.1,-0.1); 
    \draw[blue, line width=1pt, dashed] (currentBounce) foreach \dir in {1, 1, 0, 1, 0, 0}{ 
      -- ++(\dir*90:1) coordinate (currentBounce)
    };
  \end{tikzpicture}
\begin{tikzpicture}[scale=0.8] 
    \foreach \x/\y in {}{
      \fill[red!50] (\x,\y) rectangle +(1,1);
    }
    \draw[help lines] (0,0) grid (3,3); 
    \draw[dashed] (0,0) -- (3,3); 
    \coordinate (current) at (0,0);
    \foreach \dir/\length in {1/2,0/1,0/1,1/1,0/1}{
      \ifnum\dir=0 
        \draw[line width=2pt] (current) -- ++(\length,0) coordinate (current);
      \else 
        \draw[line width=2pt, -latex] (current) -- ++(0,\length) coordinate (current); 
      \fi
    }
    \coordinate (currentBounce) at (0.1,-0.1); 
    \draw[blue, line width=1pt, dashed] (currentBounce) foreach \dir in {1, 1, 0, 0, 1, 0}{ 
      -- ++(\dir*90:1) coordinate (currentBounce)
    };
  \end{tikzpicture}
    \caption{Top: The $(1,2)$-Dyck paths used to calculate $C_{(1,2)}(q,t)$. 
    Bottom: The $(2,1)$-Dyck paths used to calculate $C_{(2,1)}(q,t)$. 
    The red boxes contribute to Xin-Zhang's area and their bounce path is drawn in dashed blue. 
    The $(1,2)$-Dyck paths have area $1$ and $0$ and bounce $0$ and $1$. 
    The $(2,1)$-Dyck paths have area $2$, $1$, and $0$ and bounce $0$, $1$, and $2$. 
    Thus,~$C_{(1,2)}(q,t) = q+t$ and~$C_{(2,1)}(q,t) = q^2+qt+t^2$. 
    Summing, we obtain~$C_{\lambda = (2,1)}(q,t) = q^2+qt+t^2+q+t$.}
    \label{fig:C_[2,1](q,t)}
\end{figure}

\subsection{Integer-point generating functions of polyhedral cones}

As mentioned in the introduction, our computations employ the arithmetic of integer points in polyhedral cones, and we sketch how to compute their generating functions here, following \cite[Section~4.8]{crt}. 
A \ourdef{polyhedral cone} is a nonnegative linear combination of a finite set of vectors in $\R^d$. 
The cone is \ourdef{simplicial} if its generators are linearly independent, and \ourdef{rational} if we can choose its generators to be in $\Z^d$. 
To a rational cone $C$ we associate its \ourdef{integer-point transform}
\[
  \sigma_C (\vec z) := \sum_{ \vec m \in C \cap \Z^d } z_1^{ m_1 } z_2^{ m_2 } \cdots z_d^{ m_d } .
\]
We abbreviate the monomial in the sum as ${\vec z}^{ \vec m }$. 
We compute integer-point transforms of simplicial rational cones, with a twist, namely, some of its facets (faces of codimension 1) removed.
To be precise, fix linearly independent vectors $\vec v_1, \vec v_2, \dots, \vec v_k \in \Z^d$ and let 
\[
  C := \R_{ \ge 0 } \vec v_1 + \dots + \R_{ \ge 0 } \vec v_{m-1} + \R_{ >0 } \vec v_m + \dots
+ \R_{ >0 } \vec v_k \, ,
\]
which we call a \ourdef{half-open cone}; here $1 \le m \le k$.
Because $C$ is simplicial, each facet of $C$ is opposite one of its generators, and so we may think of $C$ as having the facets opposite $\vec v_m, \vec v_{ m+1 }, \dots, \vec v_k$ removed.
We call the vectors $\vec v_1,\dots, \vec v_k$  \ourdef{generators} of this cone.
A standard tiling argument and geometric series yield (see for example~\cite[proof of  Theorem~4.8.1]{crt}),
\[
  \sigma_C (\vec z) = \frac{ \sigma_{\Pi(C)} (\vec z) }{ \prod_{ j=1 }^k (1 - {\vec z}^{ \vec v_j }) }, 
\]
where
\[
  \Pi(C) := [0,1) \vec v_1 + \dots + [0,1) \vec v_{m-1} + (0,1] \vec v_m + \dots + (0,1] \vec v_k \, 
\]
denotes the \ourdef{fundamental parallelepiped} of $C$.
In particular, $\sigma_C (\vec z)$ is a rational function and so (because we can triangulate any cone into simplicial cones) is the integer-point transform of any rational cone.
Now, we are ready to apply these methods to our first application.


\section{The case \texorpdfstring{$\vec{k} = (k_1,k_2,k_3)$}{k = k1,k2,k3}}\label{sec:k_1k_2k_3}

Throughout this section, let~$\vec{k}=(k_1,k_2,k_3)$ be a vector of three positive integers that sum to~$n$. 
For a~$\vec{k}$-Dyck path, denote its red ranks (as defined in Definition~\ref{def:red_ranks})~$r_1,r_2,r_3$, recalling that~$r_1$ is 0. 
As discussed in Section~\ref{sec:Background}, the bounce of a~$\vec{k}$-Dyck path~$D$ is given by equation~\eqref{eq: bounce for k_1,k_2,k_3}, and 
\[
\text{area}(D)=r_2+r_3 \, .
\]
Our goal in this section is to replicate the following formula due to Xin and
Zhang~\cite[Section~2.3]{xin2022qtsymmetry} for the generating function~\eqref{eq:intro-F},
i.e., 
\begin{equation*}
\genfplain(x_1,x_2,x_3,q,t)=\sum_{k_1,k_2,k_3\ge
0}x_1^{k_1}x_2^{k_2}x_3^{k_3} \sum_{D\in\mathcal{D}_{(k_1,k_2,k_3)}}
q^{\text{area}(D)} \, t^{\text{bounce}(D)}.
\end{equation*}

\begin{theorem}\label{thm:k1k2k3}
The generating function $\genfkthree{}(x_1,x_2,x_3,q,t)$ equals
 \[
\frac{(1-x_1x_2qt^2)(1-x_1x_2q^2t)}{(1-x_2q)(1-x_2t)(1-x_1qt)(1-x_1t^2)(1-x_1q^2)(1-x_1x_2qt)(1-x_3)}
\, .
 \]
\end{theorem}

Note that this formula is $q,t$-symmetric, and so the $q,t$-symmetry of
$C_{(k_1,k_2,k_3)}(q,t)$ follows immediately.
We represent the 
$(k_1,k_2,k_3)$-Dyck paths as integer points of cones and then compute
$\genfkthree{}(x_1,x_2,x_3,q,t)$ via integer-point transforms.

Xin and Zhang~\cite{xin2022qtsymmetry} split their computation of this generating function into two parts, each divided into two cases based on numerical relationships between the quantities~$k_2$,~$r_2$, and~$r_3$ that are relevant in the computation of bounce. 
We consider both the conditions in the bounce formula and the geometry of our cones to form a coarser decomposition than that of Xin and Zhang.

Each $(k_1,k_2,k_3)$-Dyck path corresponds to a vector of the form $(k_1, k_2, k_3, r_2, r_3)$ where we omit~$r_1$, as it always equals~$0$. 
By definition the following conditions hold for~$k_1,k_2,k_3$-Dyck paths:
\begin{align}
&k_1, k_2, k_3 \geq 0 ,\label{eq: ineq1}\\
&k_1 \geq r_2 \geq 0,\label{eq: ineq2}\\
&r_2 + k_2 \geq r_3 \geq 0 \label{eq: ineq3}.
\end{align}
\begin{figure}[htbp!]
     \centering
     \begin{tikzpicture}
     
\node[circle,fill=black,minimum size=4pt,inner sep=-1,draw=black,label=above left:{$\left[\begin{smallmatrix}
         1\\0\\0\\1\\0      \end{smallmatrix}\right]$}] (1) at (-4,4) {};   

\node[circle,fill=black,minimum size=4pt,inner sep=-1,draw=black,label=above:{$\left[\begin{smallmatrix}
    1/2\\1/2\\0\\1/2\\0
\end{smallmatrix}\right]$}] (2) at (0,4) {};

\node[circle,fill=black,minimum size=4pt,inner sep=-1,draw=black,label=above right:{$\left[\begin{smallmatrix}
          0\\1\\0\\0\\0
      \end{smallmatrix}\right]$} ] (3) at (4,4) {}; 

\node[circle,fill=black,minimum size=4pt,inner sep=-1,draw=black,label=below right:{$\left[\begin{smallmatrix}
          0\\1\\0\\0\\1
      \end{smallmatrix}\right]$} ] (4) at (4,0) {}; 

\node[circle,fill=black,minimum size=4pt,inner sep=-1,draw=black,label=below:{$\left[\begin{smallmatrix}
          1/2\\1/2\\0\\1/2\\1
      \end{smallmatrix}\right]$} ] (5) at (0,0) {};

\node[circle,fill=black,minimum size=4pt,inner sep=-1,draw=black,label=below left: {$\left[\begin{smallmatrix}
    1\\0\\0\\1\\1
\end{smallmatrix}\right]$}] (6) at (-4,0) {};

\coordinate [label=$C_1$] (11) at (-3,2);
\coordinate [label=$C_2$] (12) at (3,2);
\coordinate [label=$C_3$] (13) at (0,2);

\coordinate (7) at (-3.8,0);
\coordinate (8) at (0,3.8);
\coordinate (9) at (0,3.8);
\coordinate (10) at (3.8,0);
\draw (1) to (2) to (3) to (4) to (5) to (6) to (1);
\draw (6) to (2) to (4);
\draw[dashed] (7) to (8);
\draw[dashed] (9) to (10);

     \end{tikzpicture}
     \caption{Decomposition of the~$k_1,k_2,k_3$-case illustrated via projection onto \[k_3=0, k_1+k_2=1, \text{ and } k_1=r_1.\]  
Solid and dashed lines indicate included and missing faces of the respective cones.} 
     \label{fig: open the big cone}
 \end{figure}
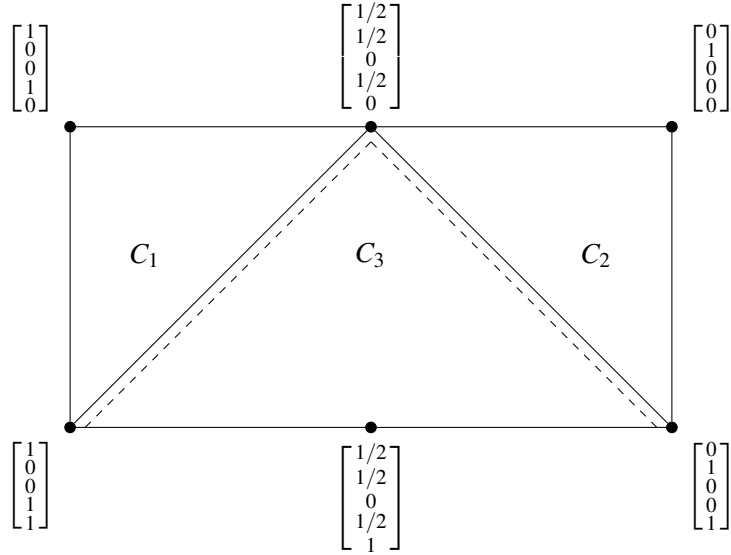
These inequalities define a 5-dimensional cone with six generators: four of
them depicted in Figure~\ref{fig: open the big cone} plus the unit vectors in direction $k_1$ and $k_3$.
This cone naturally decomposes into three simplicial cones, also shown in
Figure~\ref{fig: open the big cone}.
For each of these subcones, we write its integer-point transform
\begin{align*}
    \sigma_{C_i} (z_1, z_2, z_3, w_2, w_3) = \sum_{(\vec{k}, \vec{r}) \in C_i \cap \Z^5} z_1^{k_1} z_2^{k_2} z_3^{k_3} w_2^{r_2} w_3^{r_3}.
\end{align*}
We now compute these.

\subsection{Cone \texorpdfstring{$C_1$}{C1}} \label{subsec: cone C_1}
Let $C_1$ be the cone defined by the inequalities~\eqref{eq: ineq1}-\eqref{eq: ineq3} along with
\begin{equation*}
    r_2+k_2-r_3 \geq 2k_2 \quad \text{ and } \quad r_2 \geq k_2 \, .
\end{equation*}
Thus, $C_1$ has the generators
\[\begin{bmatrix}
    1\\0\\0\\0\\0
\end{bmatrix}, \begin{bmatrix}
    0\\0\\1\\0\\0
\end{bmatrix}, \begin{bmatrix}
    1\\0\\0\\1\\0
\end{bmatrix}, \begin{bmatrix}
    1\\0\\0\\1\\1
\end{bmatrix},\begin{bmatrix}
    1\\1\\0\\1\\0
\end{bmatrix}
\]
and is \ourdef{unimodular}: these generators are a lattice basis for $\Z^5$.
The corresponding integer-point transform is
\begin{align*}\label{eq: gf of C_1}
    \sigma_{C_1}(z_1,z_2,z_3,w_2,w_3) &=
\frac{1}{(1-z_1)(1-z_3)(1-z_1w_2)(1-z_1w_2w_3)(1-z_1z_2w_2)} \, .
\end{align*}

Any~$\vec{k}$-Dyck path corresponding to an integer point in $C_1$ satisfies both \[\min(r_2,k_2)=k_2 \text{  and  } r_2+k_2-r_3\ge 2\min(r_2,k_2),\] so its bounce is computed by
\[
\text{bounce}(D)=2k_1-r_2-r_3 
\]
and thus, the contribution of $C_1$ to \eqref{eq:intro-F} is the generating function
\begin{align*}\label{eqn:F1}
    \genfkthree_{1}(x_1,x_2,x_3,q,t) &= \sum_{(\vec{k}, \vec{r}) \in C_1 \cap \Z^5}
x_1^{k_1} x_2^{k_2}x_3^{k_3} q^{r_2 + r_3} t^{2k_1-r_2-r_3} .
\end{align*}
Writing down the generating function~$\genfkthree_{1}$ is equivalent to replacing each monomial term $z_1^{k_1}z_2^{k_2}z_3^{k_3}w_2^{r_2}w_3^{r_3}$ of~$\sigma_{C_1}$ with the corresponding term~$x_1^{k_1} x_2^{k_2}x_3^{k_3} q^{r_2 + r_3} t^{2k_1-r_2-r_3}$, and so
\begin{align*}
    \genfkthree_{1}(x_1,x_2,x_3,q,t)&= \sum_{(\vec{k}, \vec{r}) \in C_1 \cap \Z^5}
(x_1t^2)^{k_1} \left(x_2\right)^{k_2}\left(x_3\right)^{k_3}
\left(qt^{-1}\right)^{r_2} (qt^{-1})^{r_3} \\
    &=
\sigma_{C_1}(x_1t^2,x_2,x_3,qt^{-1},qt^{-1})\\&=\frac{1}{(1-x_1t^2)(1-x_3)(1-x_1qt)(1-x_1q^2)(1-x_1x_2qt)}
\, .
\end{align*}

\subsection{Cone \texorpdfstring{$C_2$}{C2}}\label{subsec: cone C_2}
Let $C_2$ be the cone defined by the inequalities~\eqref{eq: ineq1}-\eqref{eq: ineq3} along with
\begin{equation*}
    r_2+k_2-r_3 \geq 2k_2 \quad \text{ and } \quad  r_2 \leq k_2 \, .
\end{equation*}
The cone~$C_2$ is also unimodular, generated by 
\[
\begin{bmatrix}
    1\\0\\0\\0\\0
\end{bmatrix}, \begin{bmatrix}
    0\\1\\0\\0\\0
\end{bmatrix}, \begin{bmatrix}
    0\\0\\1\\0\\0
\end{bmatrix}, \begin{bmatrix}
    1\\1\\0\\1\\0
\end{bmatrix}, \begin{bmatrix}
    0\\1\\0\\0\\1
\end{bmatrix},
\]
and the corresponding integer-point transform is
\begin{align*}\label{eq: gf of C_2}
    \sigma_{C_2}(z_1,z_2,z_3,w_2,w_3) &=
\frac{1}{(1-z_1)(1-z_2)(1-z_3)(1-z_1z_2w_2)(1-z_2w_3)} \, .
\end{align*}

For any~$\vec{k}$-Dyck path~$D$ corresponding to an integer point in this cone, 
\[
\text{bounce}(D)=2(k_1-r_2)+k_2-r_3 \, ,
\]
and so the contribution of $C_2$ to \eqref{eq:intro-F} is the generating function
\begin{align*}
    \genfkthree_{2}(x_1,x_2,x_3,q,t) &= \sum_{(\vec{k}, \vec{r}) \in C_2 \cap \Z^5}
x_1^{k_1} x_2^{k_2}x_3^{k_3} q^{r_2 + r_3} t^{2(k_1-r_2)+k_2-r_3}\\
&=\sum_{(\vec{k}, \vec{r}) \in C_2 \cap \Z^5} (x_1t^2)^{k_1}
\left(x_2t\right)^{k_2}\left(x_3\right)^{k_3} \left(qt^{-2}\right)^{r_2}
(qt^{-1})^{r_3} \\
    &= \sigma_{C_2}(x_1t^2,x_2t,x_3,qt^{-2},qt^{-1}) \\
    &=\frac{1}{(1-x_1t^2)(1-x_2t)(1-x_3)(1-x_1x_2qt)(1-x_2q)} \, .
\end{align*}

\subsection{Cone \texorpdfstring{$C_3$}{C3}}\label{subsec: cone C_3}
Let $C_3$ be the cone defined by the inequalities~\eqref{eq: ineq1}-\eqref{eq: ineq3} along with
\begin{equation*}
    r_2+k_2-r_3 < 2k_2 \quad \text{ and } \quad r_2+k_2-r_3 < 2r_2 \, .
\end{equation*}
The cone~$C_3$ is generated by
\[
\begin{bmatrix}
    1\\0\\0\\0\\0
\end{bmatrix},  \begin{bmatrix}
    0\\0\\1\\0\\0
\end{bmatrix}, \begin{bmatrix}
    1\\1\\0\\1\\0
\end{bmatrix}, \begin{bmatrix}
    1\\0\\0\\1\\1
\end{bmatrix},\begin{bmatrix}
    0\\1\\0\\0\\1
\end{bmatrix}
\]
and thus, not unimodular. 
The fundamental parallelepiped of $C_3$ is 
 \[\Pi(C_3) = \left\{\lambda_1 \begin{bmatrix}
    1\\0\\0\\0\\0
\end{bmatrix} + \lambda_2 \begin{bmatrix}
    0\\0\\1\\0\\0
\end{bmatrix} + \lambda_3\begin{bmatrix}
    1\\1\\0\\1\\0
\end{bmatrix} + \lambda_4 \begin{bmatrix}
    1\\0\\0\\1\\1
\end{bmatrix}+\lambda_5\begin{bmatrix}
    0\\1\\0\\0\\1
\end{bmatrix} \ \middle|  \ 0 \leq \lambda_1,\lambda_2,\lambda_3 < 1 \text{ and } 0< \lambda_4,\lambda_5 \leq 1\right\},\]
with $\Pi(C_3) \cap \Z^5 = \{(1,1,0,1,2), (1,1,0,1,1)\}$. 
Thus,
\begin{align}
        \sigma_{C_3}(z_1,z_2,z_3,w_2,w_3)= &\frac{ z_1 z_2 w_2 w^2_3 + z_1 z_2
w_2w_3}{(1-z_1)(1-z_3)(1-z_1z_2w_2)(1-z_1w_2w_3)(1-z_2w_3)} \, .
\label{eq:c3intpttransf}
\end{align}

For the cones~$C_1$ and~$C_2$, we used the linearity of the bounce formulas to write~$\genfkthree_{1},\genfkthree_{2}$ using a change of variables in~$\sigma_{C_1},\sigma_{C_2}$. 
However, for a $\vec{k}$-Dyck path~$D$ corresponding to an integer point in~$C_3$, 
\[
\text{bounce}(D)=  2(k_1-r_2)+\left\lceil\frac{r_2+k_2-r_3}{2}\right\rceil .
\]
In order to obtain a rational formula for~$\genfkthree_{3}$, we can partition the integer points in $C_3$ into two subsets, such that bounce is a linear function on each subset.
This partition corresponds to the two summands in the numerator
of~\eqref{eq:c3intpttransf}.
Let
\[
A = \left\{
\alpha_1\begin{bmatrix}
    1\\0\\0\\0\\0
\end{bmatrix}  +
\alpha_2\begin{bmatrix}
    0\\0\\1\\0\\0
\end{bmatrix}+
\alpha_3\begin{bmatrix}
    1\\1\\0\\1\\0
\end{bmatrix}+
\alpha_4\begin{bmatrix}
    1\\0\\0\\1\\1
\end{bmatrix}+
\alpha_5\begin{bmatrix}
    0\\1\\0\\0\\1
\end{bmatrix} \\\ \middle| \ \alpha_1,\alpha_2,\alpha_3\in\Z_{\ge 0} \text{ and } \alpha_4,\alpha_5\in\Z_{>0}
\right\}.
\]
For the points in $A$, area is computed by $\alpha_3+2\alpha_4+\alpha_5$, and (since~$A$ falls under the second bounce condition) bounce is computed by
 \begin{align*}
     2(k_1-r_2)+\left\lceil\frac{r_2+k_2-r_3}{2}\right\rceil&=
     2(\alpha_1+\alpha_3+\alpha_4-\alpha_3-\alpha_4)+\left\lceil\frac{\alpha_3+\alpha_4+\alpha_3+\alpha_5-\alpha_4-\alpha_5}{2}\right\rceil\\&=2\alpha_1+\alpha_3
\, .
 \end{align*}
Hence, $A$ contributes the generating function
\begin{align*}
  \genfkthree_{3A}(x_1,x_2,x_3,q,t)
  &= \sum_{\alpha_i \in \mathbb{Z}_{\geq 0}, \, \alpha_4,\alpha_5\neq
0}(x_1t^2)^{{\alpha_1}} (x_3)^{\alpha_2} (x_1x_2qt)^{\alpha_3}
(x_1q^2)^{\alpha_4} (x_2q)^{\alpha_5} \\
  &= \frac{  x_1x_2q^3}{(1-x_1t^2)(1-x_3)(1-x_1x_2qt)(1-x_2q)(1-x_1q^2)} \, .
\end{align*}

The second set of the partition of integer points in $C_3$ is
\[
B = \left\{\begin{bmatrix}
    1\\1\\0\\1\\1
\end{bmatrix} + 
\alpha_1\begin{bmatrix}
    1\\0\\0\\0\\0
\end{bmatrix}  +
\alpha_2\begin{bmatrix}
    0\\0\\1\\0\\0
\end{bmatrix}+
\alpha_3\begin{bmatrix}
    1\\1\\0\\1\\0
\end{bmatrix}+
\alpha_4\begin{bmatrix}
    1\\0\\0\\1\\1
\end{bmatrix}+
\alpha_5\begin{bmatrix}
    0\\1\\0\\0\\1
\end{bmatrix}  \ \middle| \  \alpha_i\in\Z_{\ge 0}
\right\}.
\]
Area is computed by~$\alpha_3+2\alpha_4+\alpha_5 + 2$, by summing the fourth and fifth coordinates corresponding to~$r_2$ and~$r_3$, and the constant 2 represents the contribution from the special point~$(1,1,0,1,1)$.
Since~$B$ falls under the second bounce condition, bounce is computed by
 \begin{align*}
     2(k_1-r_2)+\left\lceil\frac{r_2+k_2-r_3}{2}\right\rceil=  2\alpha_1 +
\alpha_3 + 1 \, ,
 \end{align*} 
and so the generating function enumerating points in set~$B$ is
\[ \genfkthree_{3B}(x_1,x_2,x_3,q,t) =\frac{
x_1x_2q^2t}{(1-x_1t^2)(1-x_3)(1-x_1x_2qt)(1-x_2q)(1-x_1q^2)} \, . \]

Lastly, we combine the generating functions for sets~$A$ and~$B$ to obtain a generating function for~$(k_1,k_2,k_3)$-Dyck paths in cone~$C_{3}$:
\[\genfkthree_{3}(x_1,x_2,x_3,q,t) =\frac{ x_1x_2q^3 + x_1x_2 q^2
t}{(1-x_1t^2)(1-x_3)(1-x_1x_2qt)(1-x_2q)(1-x_1q^2)} \, .\]

\subsection{Summing over the full cone}
By construction, any~$(k_1,k_2,k_3)$-Dyck path corresponds to a point in either~$C_1, C_2$ or~$C_3$. 
The only points that appear in more than one cone are points in the cone generated by ~$\left[\begin{smallmatrix}
     1\\0\\0\\0\\0
 \end{smallmatrix}\right],\left[\begin{smallmatrix}
     0\\0\\1\\0\\0
 \end{smallmatrix}\right],\left[\begin{smallmatrix}
     1\\1\\0\\1\\0
 \end{smallmatrix}\right]$
(represented by the point labeled by $\left[\begin{smallmatrix}
    1/2\\1/2\\0\\1/2\\0
\end{smallmatrix}\right]$
in Figure~\ref{fig: open the big cone}), which appear in both cones~$C_1$ and~$C_2$ and are enumerated by the generating function
 \[ \sum_{(\vec{k}, \vec{r}) \in (C_1\cap C_2)\cap \Z^5} x_1^{k_1}
x_2^{k_2}x_3^{k_3} q^{r_2 + r_3} t^{2k_1-r_2-r_3} =
\frac{1}{(1-x_1t^2)(1-x_3)(1-x_1x_2qt)} \, . \]
Hence,
 \begin{align*}
     \genfkthree{}(x_1,x_2,x_3,q,t) &= \genfkthree_{1} + \genfkthree_{2} + \genfkthree_{3}
-\frac{1}{(1-x_1t)(1-x_3)(1-x_1x_2qt)} \\ &=
\frac{(1-x_1x_2qt^2)(1-x_1x_2q^2t)}{(1-x_2q)(1-x_2t)(1-x_1qt)(1-x_1t^2)(1-x_1q^2)(1-x_1x_2qt)(1-x_3)}
\, ,
 \end{align*} 
and this finishes our proof of Theorem~\ref{thm:k1k2k3}.


\section{The case \texorpdfstring{$\vec{k} = (k,k,k,k)$}{k4}}\label{sec:k^4}

We turn our attention to our second application, where we revisit ideas
from the previous section, as well as introduce some new ones.
Our goal in this section is to compute~\eqref{eq:intro-H}, i.e.,
\[
  \genffourk(x, q, t) := \sum_{ k \in \Z_{ \ge 0 } } C_{(k,k,k,k)}(q,t) \, x^k,
\]
re-establishing the following formula of Xin and Zhang~\cite{xin2022qtsymmetry}. 

\begin{theorem}\label{thm:kkkk}
The generating function $\genffourk(x,q,t)$ equals
\begin{equation*}
\frac{N}{(1-q^3tx)(1-qt^3x)(1-q^2t^2x)(1-q^6x)(1-t^6x)} \, ,
\end{equation*}
where 
\begin{align*}
    N=1+& \left(q^5t+qt^5+q^4t^2+q^2t^4+q^4t+qt^4+q^3t^2+q^2t^3+q^3t^3\right)x\\
    & {} + \left(-q^7t^3-q^3t^7+q^6t^5+q^5t^6-q^6t^4-q^4t^6-q^5t^5-q^5t^4-q^4t^5\right)x^2\\
    & {} - \left(q^8t^8+q^9t^6+q^6t^9+q^8t^7+q^7t^8\right)x^3 .
\end{align*}
\end{theorem}

Note again the~$q,t$-symmetry of~$\genffourk$.
Throughout this section, let~$\vec{k}=(k,k,k,k)$, where $k$ is a positive integer. 
Recall that a~$(k,k,k,k)$-Dyck path~$D$ can be uniquely identified with the
vector $(k,a,b,c)$ based on the red ranks~$(0, r_2, r_3,r_4)$, where
\[
r_2 = k-a, \quad r_3 = 2k-a-b, \quad 
\text{ and } r_4 = 3k-a-b-c.
\]
We have from ~\cite{Niu} that $\operatorname{area}(D) = 6k-3a-2b-c$ and
\begin{align*}
    \text{bounce}(D) &= \begin{cases}
        6a + 3b + c - 4k & \text{ if } b \geq 2k-2a \text{ and } c\geq 4k-2a-2b, \\
        5a + 2b + \lceil\frac{c}{2}\rceil - 2k  & \text{ if } b \geq 2k-2a \text{ and } c< 4k-2a-2b,\\
        4a + 2b + c -2k  & \text{ if } b < 2k-2,\, b \text{ is even, and }c\geq 3k-a-\frac{3b}{2},\\ 
        2a + \frac{b}{2} + k + \left\lceil \frac{3a +\frac{3b}{2}+c - 3k}{2}
\right\rceil & \text{ if } b < 2k-2,\, b \text{ is even, and } \\
        & \quad \quad \quad 3k-3a-\frac{3b}{2} \leq c<3k-a-\frac{3b}{2},\\ 
        3a + b + \frac{c}{3} & \text{ if } b < 2k-2a,\, b \text{ is even, and } c< 3k-3a-\frac{3b}{2},\\ 
        4a + 2b + c-2k + 1 & \text{ if } b < 2k-2a,\, b \text{ is odd, and } \\
        & \quad \quad \quad c\geq 3k-a-\frac{3(b+1)}{2} +1, \\ 
        2a + \frac{b+1}{2} + k +  &\text{ if } b < 2k-2a,\, b \text{ 
is odd, and }\\
\quad \quad \frac{3a + \frac{3(b+1)}{2} + c -3k-1 }{2} & 3k-3a - \frac{3(b+1)}{2}+1\leq c<3k-a-\frac{3(b+1)}{2} +1 , \\
        3a + b + 1 + \lceil \frac{c-1}{3}\rceil &\text{ if } b < 2k-2, \, b \text{ is odd, and } \\
        & \quad \quad \quad c< 3k-3a - \frac{3(b+1)}{2}+1.    \end{cases}
\end{align*}

We now use the techniques from Section~\ref{sec:k_1k_2k_3} to compute the generating function
\[\genffourk(x,y_1,y_2,y_3,q,t) = \sum_{k\geq 0 } x^k \sum_{(k,a,b,c) \in \mathcal{D}_{(k,k,k,k)}} y_1^a \  y_2^b \ y_3^c \ q^{\text{area}(D)}  \ t^{\text{bounce}(D)},\]
as a sum of generating functions derived from cones. 
By definition, all vectors $(k,a,b,c)$ corresponding to~$(k,k,k,k)$-Dyck paths must satisfy the following inequalities:
\begin{align}
    0 &\leq a \leq k \label{ineq:k^4 1} \, ,\\
    0 &\leq b \leq 2k-a\label{ineq:k^4 2} \, ,\\
    0 &\leq c \leq 3k-a-b\label{ineq:k^4 3} \, .
\end{align}
This system of inequalities corresponds to the cone~$C$ generated by 
\begin{align*}\label{k^4_generators}
    \begin{bmatrix}1\\0\\2\\0
\end{bmatrix}, \begin{bmatrix}1\\0\\0\\0
\end{bmatrix}, \begin{bmatrix}1\\1\\0\\0
\end{bmatrix}, \begin{bmatrix}1 \\ 1\\1\\0 \end{bmatrix},\begin{bmatrix} 1 \\ 1\\1\\1 \end{bmatrix},  \begin{bmatrix}1 \\ 0\\2\\1 \end{bmatrix}, \begin{bmatrix}1 \\ 0\\0\\3\end{bmatrix},  \begin{bmatrix}1 \\ 1\\0\\2\end{bmatrix}.
\end{align*}
We denote these cone generators by~$v_1, v_2, \dots, v_8$, respectively.
Figure~\ref{fig: k4 generators} depicts a projection of them to the hyperplane
$k=1$.
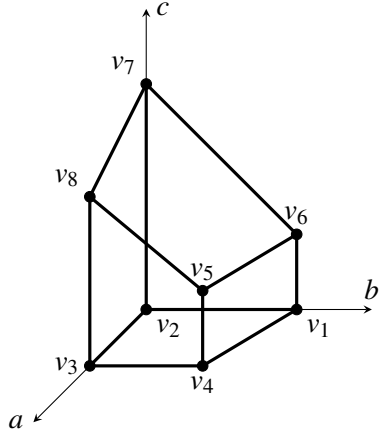
\begin{figure}[htbp!]
    \centering
\begin{tikzpicture}
        \coordinate [label=below right:$v_1$] (v1) at (2,0);
        \circleat{2}{0}{black};
        
        \coordinate [label=below right:$v_2$] (v2) at (0,0);
        \circleat{0}{0}{black};
        
        \coordinate [label=left:$v_3$] (v3) at (-.75,-.75);
        \circleat{-.75}{-.75}{black};
        
        \coordinate [label=below:$v_4$] (v4) at (.75,-.75);
        \circleat{.75}{-.75}{black};

        \coordinate [label=above:$v_5$] (v5) at (0.75,0.25);
        \circleat{0.75}{0.25}{black};

        \coordinate [label=above:$v_6$] (v6) at (2,1);
        \circleat{2}{1}{black};

        \coordinate [label=above left:$v_7$] (v7) at (0,3);
        \circleat{0}{3}{black};

        \coordinate [label=above left:$v_8$] (v8) at (-.75,1.5);
        \circleat{-.75}{1.5}{black};

\coordinate [label=left:$a$] (x) at (-1.5,-1.5);
\coordinate [label=above:$b$] (y) at (3,0);
\coordinate [label=right:$c$] (z) at (0,4);

\draw[very thick] (v3) -- (v8) -- (v7) -- (v6) -- (v1) -- (v4) -- (v3);
\draw[very thick] (v8) -- (v5) -- (v6);
\draw[very thick] (v4) -- (v5);
\draw[very thick] (v2) -- (v7);
\draw[very thick] (v2) -- (v3);
\draw[very thick] (v2) -- (v1);

\draw[-stealth] (v2) to (x);
\draw[-stealth] (v2) to (y);
\draw[-stealth] (v2) to (z);

    \end{tikzpicture}
    \caption{All generators have first coordinate 1, so we visualize them by projecting to the last three coordinates.}
    \label{fig: k4 generators}
\end{figure}
To compute the integer-point transform
\[\sigma_{C}(y,z_1,z_2,z_3) = \sum\limits_{(k,a,b,c)\in C} y^k z_1^a z_2^b
z_3^c \, , \]
we triangulate $C$ depending on the bounce cases listed above.
We follow the notation in \cite{xin2022qtsymmetry} and denote
\begin{align*}
    \textbf{Part 1}&:b \geq 2k-2a,\\\textbf{Part 2}&: b < 2k-2a \text{ and~$b$ even}, \\ \textbf{Part 3}&: b < 2k-2a \text{ and~$b$ odd}.
\end{align*}
Part 1 is split into two cases, 
Part 2 into three cases where~$b$ is always even,
and Part 3 into three cases where~$b$ is always odd.

\subsection*{Part 1}\label{subsec:part1}
We divide the points satisfying $b \ge 2k - 2a$ into two cases:
\begin{align}
    \textbf{Case 1: }& c\geq 4k-2a-2b,\label{ineq:k^4 part 1.1}\\\textbf{Case 2: }& c< 4k-2a-2b.\label{ineq:k^4 part 1.2}
\end{align}

Let $C_{1.1}$ be the cone defined by the inequalities
~\eqref{ineq:k^4 1}-\eqref{ineq:k^4 3} and \eqref{ineq:k^4 part 1.1}; it has
generators~$v_1, v_4, v_5, v_6,v_8$ and is
not simplicial, but can be triangulated into two simplicial cones:
\begin{itemize}
    \item  Cone A generated by $v_1, v_4, v_6,v_8$, \item Cone B generated by $v_4, v_5, v_6,v_8.$
\end{itemize}
We remove the facet opposite~$v_5$ in cone~$B$; it equals the facet opposite~$v_1$ in cone~$A$ which is included there (see Figure~\ref{fig: part 1 case 1 and 2}).
\begin{figure}[htbp!]
    \centering

    \begin{tikzpicture}

\begin{scope}[on background layer]
\coordinate [label=right:$c$] (z) at (0,4);
\coordinate [label=above left:$v_7$] (v7) at (0,3);
        \circleat{0}{3}{black};
\draw[-stealth] (v7) to (z);
\draw[very thick] (v2) -- (v7);
\draw[very thick] (v2) -- (v1);
\draw[-stealth, very thick] (v2) to (y);

\end{scope}

        \coordinate [label=below right:$v_1$] (v1) at (2,0);
        \circleat{2}{0}{magenta};
        
        \coordinate [label=left:$v_2$] (v2) at (0,0);
        \circleat{0}{0}{black};
        
        \coordinate [label=left:$v_3$] (v3) at (-.75,-.75);
        \circleat{-.75}{-.75}{black};
        
        \coordinate [label=below:$v_4$] (v4) at (.75,-.75);
        \circleat{.75}{-.75}{magenta};

        \coordinate [label=above:$v_6$] (v6) at (2,1);
        \circleat{2}{1}{magenta};

        \coordinate [label=above left:$v_8$] (v8) at (-.75,1.5);
        \circleat{-.75}{1.5}{magenta};
        \filldraw[magenta!50] (v4) -- (v6) -- (v8) --(v4);

\filldraw[magenta!20] (v4) -- (v6) -- (v1) -- (v4);
        \coordinate [label=above:$v_5$] (v5) at (0.75,0.25);
        \circleat{0.75}{0.25}{magenta};

\coordinate [label=left:$a$] (x) at (-1.5,-1.5);
\coordinate [label=above:$b$] (y) at (3,0);

\draw[very thick] (v3) -- (v8) -- (v7) -- (v6);
\draw[very thick] (v4) -- (v3);
\draw[very thick] (v2) -- (v3);

\draw[magenta, very thick] (v1) -- (v6) -- (v5) -- (v4) -- (v1);
\draw[magenta, very thick] (v4) -- (v8) -- (v6);
\draw[magenta, very thick] (v5) -- (v8);

\draw[magenta, very thick] (v4) --  (v6);

\draw[-stealth] (v2) to (x);

\coordinate [label=right:B] (B) at (2,3);
\coordinate (Bloc) at (1,1);
\draw (Bloc) to [out=90,in=180] (B) ;

\coordinate [label=right:A] (A) at (3,1);
\coordinate (Aloc) at (1.8,0.3);
\draw (Aloc) to [out=0,in=180] (A) ;

    \end{tikzpicture}
        \begin{tikzpicture}

\begin{scope}[on background layer]

\draw[-stealth] (v2) to (z);
\draw[very thick] (v2) -- (v7);
\draw[very thick] (v2) -- (v1);
\draw[-stealth] (v2) to (y);
\draw[-stealth] (v2) to (x);
\draw[very thick] (v2) -- (v3);
\draw[very thick] (v8) -- (v7) -- (v6);

\filldraw[teal!50] (v3) -- (v4) --(v8) -- (v3);
\filldraw[teal!50] (v1)-- (v4) --(v8) -- (v1);
\filldraw[pattern color=teal, pattern=crosshatch dots] (v1)-- (v4) --(v8) -- (v1);

\end{scope}

        \coordinate [label=below right:$v_1$] (v1) at (2,0);
        \circleat{2}{0}{teal};

        \coordinate [label=left:$v_3$] (v3) at (-.75,-.75);
        \circleat{-.75}{-.75}{teal};
        
        \coordinate [label=below:$v_4$] (v4) at (.75,-.75);
        \circleat{.75}{-.75}{teal};

        \coordinate [label=above:$v_5$] (v5) at (0.75,0.25);
        \circleat{0.75}{0.25}{black};

        \coordinate [label=above:$v_6$] (v6) at (2,1);
        \circleat{2}{1}{black};

        \coordinate [label=above left:$v_7$] (v7) at (0,3);
        \circleat{0}{3}{black};

        \coordinate [label=above left:$v_8$] (v8) at (-.75,1.5);
        \circleat{-.75}{1.5}{teal};

\coordinate [label=left:$a$] (x) at (-1.5,-1.5);
\coordinate [label=above:$b$] (y) at (3,0);
\coordinate [label=right:$c$] (z) at (0,4);

\draw[teal, very thick] (v3) -- (v8);
\draw[teal, very thick] (v4) -- (v3);

\draw[very thick] (v1) -- (v6) -- (v5) --(v4);
\draw[teal, very thick, dashed] (v4) -- (v1);
\draw[teal, very thick, dashed] (v4) -- (v8);

\draw[very thick] (v5) -- (v8);
\draw[teal, very thick, dashed] (v1) -- (v8);

    \end{tikzpicture}
    \caption{Part 1, Case 1 (left) and Case 2 (right).  
    In Part 1, Case 1, the shared facet for cones A and B is missing in cone B and included in cone A.
    The facet shared by the Part 1, Case 1 cone A and Part 1, Case 2 is included in cone A and missing in Part 1, Case 2.}
    \label{fig: part 1 case 1 and 2}
\end{figure}
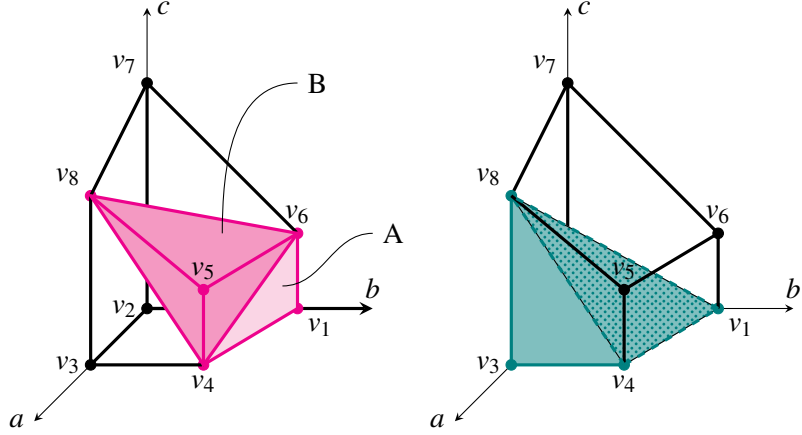

Both Cone A and Cone B are  unimodular, so we can directly compute the integer-point transforms 
\begin{align*}
    \sigma_{C_{1.1A}}(y,z_1,z_2,z_3) &=
\frac{1}{(1-yz_2^2)(1-yz_1z_2)(1-yz_2^2z_3)(1-yz_1z_3^2)} \, , \\
\sigma_{C_{1.1B}}(y,z_1,z_2,z_3) &=
\frac{yz_1z_2z_3}{(1-yz_1z_2)(1-yz_1z_2z_3)(1-yz_2^2z_3)(1-yz_1z_3^2)} \, ,
\end{align*}
which in turn gives
\begin{align*}
    \sigma_{C_{1.1}}(y,z_1,z_2,z_3) = \sigma_{C_{1A}} + \sigma_{C_{1B}}= \frac{1-y^2z_1z_2^3z_3}{(1-yz_2^2)(1-yz_1z_2)(1-yz_2^2z_3)(1-yz_1z_3^2)(1-yz_1z_2z_3)} \, .
\end{align*}
Using methods similar to the ones for first two cones in the Section
\ref{sec:k_1k_2k_3}, we compute the generating function
\begin{align*}
    \genffourk_{1.1}(x,y_1,y_2,y_3,q,t)  =
\frac{1-x^2y_1y_2^3y_3q^2t^8}{(1-xy_2^2q^2t^2)(1-xy_1y_2qt^5)(1-xy_2^2y_3qt^3)(1-xy_1y_3^2qt^4)(1-xy_1y_2y_3t^6)}
\, .
\end{align*}

Next, define $C_{1.2}$ to be the cone defined by inequalities ~\eqref{ineq:k^4
1}-\eqref{ineq:k^4 3} and \eqref{ineq:k^4 part 1.2}; it has generators~$v_1,
v_3, v_4, v_8$, and we remove the facet opposite~$v_3$.
This cone is simplicial (though not unimodular), with integer-point transform
\begin{align*}
    \sigma_{C_{1.2}}(y,z_1,z_2,z_3) &= \frac{yz_1 +
yz_1z_3}{(1-yz_2^2)(1-yz_1)(1-yz_1z_2)(1-yz_1z_3^2)} \, ,
\end{align*}
as~$\Pi(C_{1.2}) \cap \Z^4  =\{(1,1,0,0),(1,1,0,1) \}$. 
We apply the techniques of Section~\ref{subsec: cone C_3} to compute
\begin{align*}
    \genffourk_{1.2}(x,y_1,y_2,y_3,q,t)  = \frac{xy_1q^3t^3 +
xy_1y_3q^2t^4}{(1-xy_2^2q^2t^2)(1-xy_1q^3t^3)(1-xy_1y_2qt^5)(1-xy_1y_3^2qt^4)}
\, .
\end{align*}

\subsection*{Part 2}
We divide the points satisfying $b < 2k-2a$ with $b$ even into three cases:
\begin{align}
    \textbf{Case 1: }& c\geq 3k-a-\frac{3b}{2} \, ,\label{ineq:k^4 part
2.1}\\\textbf{Case 2: }&  3k-3a-\frac{3b}{2}\leq
c<3k-a-\frac{3b}{2} \, ,\label{ineq:k^4 part 2.2}\\ \textbf{Case 3: } & c<
3k-3a-\frac{3b}{2}\label{ineq:k^4 part 2.3} \, .
\end{align}

\subsection*{Part 3}
Similarly, we divide the points satisfying $b < 2k-2a$ with $b$ odd into three cases:
\begin{align}
    \textbf{Case 1: }& c\geq 3k-a-\frac{3(b+1)}{2} + 1 \, ,\label{ineq:k^4 part 3.1}\\\textbf{Case 2: }&  3k-3a-\frac{3(b+1)}{2}+1 \leq c<3k-a-\frac{3(b+1)}{2}+1 \, ,\label{ineq:k^4 part 3.2}\\ \textbf{Case 3: } & c< 3k-3a-\frac{3(b+1)}{2}+1\label{ineq:k^4 part 3.3} \, .
\end{align}

We treat each case for Parts 2 and 3 in parallel. 
We start by disregarding the parity condition on $b$ and consider all points that satisfy the given inequalities. 
Once we determine the relevant cone generating function, we extract the points
with the correct parity condition for~$b$.

\subsection*{Case 1}\label{subsec:part2,3_1}
We start with Part 2, Case 1. 
Let $C_{2.1}$ be the cone defined by the inequalities ~\eqref{ineq:k^4
1}-\eqref{ineq:k^4 3} and \eqref{ineq:k^4 part 2.1}; it has generators~$v_1,
v_6,v_7, v_8$, and the facet opposite~$v_7$ is missing.
\begin{figure}[htbp!]
    \centering
        \begin{tikzpicture}

\begin{scope}[on background layer]

\draw[-stealth] (v2) to (z);
\draw[very thick] (v2) -- (v7);
\draw[very thick] (v2) -- (v1);
\draw[-stealth] (v2) to (y);
\draw[-stealth] (v2) to (x);
\draw[very thick] (v2) -- (v3);

\filldraw[Mulberry!50] (v1) --(v6)--(v7) -- (v8) -- (v1);
\filldraw[pattern color=Mulberry, pattern=crosshatch dots] (v1) -- (v6) -- (v8) -- (v1);

\end{scope}

        \coordinate [label=below right:$v_1$] (v1) at (2,0);
        \circleat{2}{0}{Mulberry};
        
        \coordinate [label=left:$v_2$] (v2) at (0,0);
        \circleat{0}{0}{black};
        
        \coordinate [label=left:$v_3$] (v3) at (-.75,-.75);
        \circleat{-.75}{-.75}{black};
        
        \coordinate [label=below:$v_4$] (v4) at (.75,-.75);
        \circleat{.75}{-.75}{black};

        \coordinate [label=left:$v_5$] (v5) at (0.75,0.25);
        \circleat{0.75}{0.25}{black};

        \coordinate [label=above:$v_6$] (v6) at (2,1);
        \circleat{2}{1}{Mulberry};

        \coordinate [label=above left:$v_7$] (v7) at (0,3);
        \circleat{0}{3}{Mulberry};

        \coordinate [label=above left:$v_8$] (v8) at (-.75,1.5);
        \circleat{-.75}{1.5}{Mulberry};

\coordinate [label=left:$a$] (x) at (-1.5,-1.5);
\coordinate [label=above:$b$] (y) at (3,0);
\coordinate [label=right:$c$] (z) at (0,4);

\draw[very thick] (v3) -- (v8);
\draw[very thick] (v4) -- (v3);
\draw[Mulberry, very thick] (v8) -- (v7) -- (v6) -- (v1)--(v8);

\draw[very thick] (v6) -- (v5);
\draw[Mulberry, very thick] (v8) -- (v6);

\draw[very thick] (v5) --(v4);
\draw[very thick] (v4) -- (v1);

\draw[very thick] (v5) -- (v8);
\draw[Mulberry, very thick] (v1) -- (v8);

    \end{tikzpicture}
    \caption{Part 2, Case 1 and Part 3, Case 1 have the same generators but a
different apices.}
    \label{fig: part 1 case 1}
\end{figure}
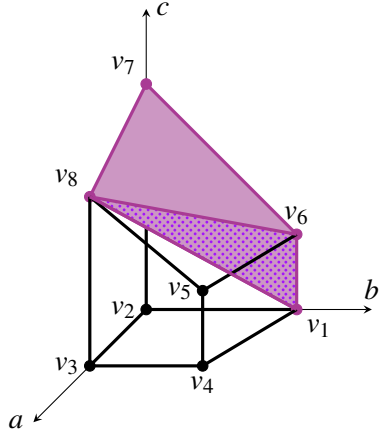
Its fundamental parallelepiped contains the two integer points~$\Pi(C_{2.1}) \cap \Z^4
=\{(1,0,0,3),(1,0,1,2) \}$ and so the corresponding integer-point transform is
\begin{align*}
    \sigma_{C_{2.1}}(y,z_1,z_2,z_3) &= \frac{yz_3^3 + yz_2z_3^2}{(1-yz_2^2)(1-yz_2^2z_3)(1-yz_3^3)(1-yz_1z_3^2)} \, .
    \end{align*}

We are interested in extracting all points~$(k,a,b,c)$ enumerated by this generating function such that~$b$ is even.
Since the coordinate~$b$ of a given point corresponds to the power of~$z_2$ in the monomial, this is equivalent to considering monomial terms where~$z_2$ is raised to an even power.
Since $z_2$ appears only with even powers in the denominator, 
the two numerator terms of~$\sigma_{C_{2.1}}$ correspond to generating functions enumerating points~$(k,a,b,c)$ where~$b$ is even and odd, respectively. 
Since we are only interested in the case that $b$ is even, we simply disregard the term $y z_2 z_3^2$ in the numerator.
It remains to apply the techniques of Sections~\ref{subsec: cone C_1} and~\ref{subsec: cone C_2} to compute the generating function
\begin{align*}
    \genffourk_{2.1}( x,y_1,y_2,y_3,q,t)  &=
\frac{xy_3^3q^3t}{(1-xy_2^2q^2t^2)(1-xy_2^2y_3qt^3)(1-xy_3^3q^3t)(1-xy_1y_3^2qt^4)}
\, .
\end{align*}

Next, we consider Part 3, Case 1, where we eventually require that $b$ is odd.
The inequalities  ~\eqref{ineq:k^4 1}-\eqref{ineq:k^4 3} and \eqref{ineq:k^4 part 3.1} reduce to the following system of inequalities:
\begin{align}
    &0  \leq a\label{ineq: 1},\\ 
    &0 \leq b \label{ineq: 2},\\ 
    &0 \leq c \label{ineq: 3},\\
    &c \leq 3k-a-b \label{ineq: 4},\\
    &b < 2k-2a \label{ineq: 5},\\ 
    &3k-a-\frac{3(b+1)}{2} + 1  \leq c. \label{ineq: 6}
\end{align}
The last inhomogeneous inequality makes computations more subtle.
For example, ~\eqref{ineq: 1},~\eqref{ineq: 5}, and~\eqref{ineq: 6} imply 
\[-\frac{1}{2}\leq 2a -\frac{1}{2} \leq 3k -3k -a + 3a -\frac{1}{2} \leq 3k-a -\frac{3}{2}b -\frac{1}{2}\leq c,\]
but since~$c\in\Z$ the implication is the redundant assumption that~$0\leq c$.

Inequalities~\eqref{ineq: 4} and~\eqref{ineq: 6}  imply~$-1 \leq b$, which is weaker than~\eqref{ineq: 2}. 
To simplify our computations, we first assume~$-1\leq b$ and compute a cone generating function, and then impose the condition~$0\leq b$ by eliminating the points where~$b = -1$.
Hence, we consider the cone~$C_{3.1}$ defined by the following minimal hyperplane description:
\begin{align*}
    &0  \leq a,\\ 
    &c \leq 3k-a-b,\\
    &b < 2k-2a,\\ 
    &3k-a-\frac{3(b+1)}{2} + 1  \leq c.
\end{align*}
The cone~$C_{3.1}$ has generators~$v_1, v_6, v_7, v_8$, and the facet opposite~$v_7$ is missing. 
Though these are the same generators as of $C_{2.1}$, the cone~$C_{3.1}$ is a translation of~$C_{2.1}$ such that its apex is at~$(-\frac{1}{2}, 0, -1, -\frac{1}{2})$ rather than the origin.  
With ~$\Pi(C_{3.1}) \cap \Z^4 = \{ (0, 0, -1 , 1), (1 , 0, 0, 2)\}$, we compute
\begin{align*}
    \sigma_{C_{3.1}}(y,z_1,z_2,z_3)  &= \frac{z_2^{-1}z_3  + y
z_3^2}{(1-yz_2^2)(1-yz_2^2z_3)(1-yz_3^3)(1-yz_1z_3^2)} \, .
    \end{align*}
To impose the condition $0 \le b$, we first subtract all terms where~$b = -1$, i.e.,~$z_2^{ -1 }$:
\begin{align*}\label{genf:3.1}
    \sigma_{C_{3.1}}(y,z_1,z_2,z_3) - \frac{z_2^{-1} z_3}{(1-yz_3^3)(1-yz_1z_3^2)} &= \frac{yz_3^2 - y^2z_2^3z_3^2  + yz_2z_3 + yz_2z_3^2}{(1-yz_2^2)(1-yz_2^2z_3)(1-yz_3^3)(1-yz_1z_3^2) }.
    \end{align*}
Finally, we need to extract the terms for odd $b$, i.e., the monomials with odd powers of~$z_2$.
Again, the denominator contributes only even powers of $z_2$, and so we simply need to disregard the first term in the numerator, giving rise to the resulting generating function
\begin{align*}
   \frac{- y^2z_2^3z_3^2 + yz_2z_3 + yz_2z_3^2
}{(1-yz_2^2)(1-yz_2^2z_3)(1-yz_3^3)(1-yz_1z_3^2)} \, .\end{align*}
From this, we obtain the contribution
\begin{align*}
    \genffourk_{3.1}(x,y_1,y_2,y_3,q,t)  &= \frac{xy_2y_3q^2 t^2 ( -xy_2^2y_3q^2t^3 + q + y_3t)
}{(1-xy_2^2q^2t^2)(1-xy_2^2y_3qt^3)(1-xy_3^3q^3t)(1-xy_1y_3^2qt^4)} \, .
\end{align*}

\subsection*{Case 2}\label{subsec: case 2}
We next consider Part 2, Case 2.
Let~$C_{2.2}$ be the cone defined by the inequalities  ~\eqref{ineq:k^4 1}-\eqref{ineq:k^4 3} and \eqref{ineq:k^4 part 2.2}; it has generators~$v_1, v_3, v_7, v_8$, and we remove the facets opposite to~$v_3, v_7$.
\begin{figure}[htbp!]
    \centering
        \begin{tikzpicture}

\begin{scope}[on background layer]

\draw[-stealth] (v2) to (z);
\draw[very thick] (v2) -- (v7);
\draw[very thick] (v2) -- (v1);
\draw[-stealth] (v2) to (y);
\draw[-stealth] (v2) to (x);
\draw[very thick] (v2) -- (v3);

        \coordinate [label=left:$v_2$] (v2) at (0,0);
        \circleat{0}{0}{black};

\filldraw[Dandelion!50] (v1) --(v3)--(v8) -- (v7) -- (v1);

\draw[Dandelion, very thick] (v1) -- (v8);

\filldraw[pattern color=Dandelion, pattern=crosshatch dots] (v1) -- (v3) -- (v8) -- (v7) -- (v1);

\end{scope}

        \coordinate [label=below right:$v_1$] (v1) at (2,0);
        \circleat{2}{0}{Dandelion};
        
        \coordinate [label=left:$v_3$] (v3) at (-.75,-.75);
        \circleat{-.75}{-.75}{Dandelion};
        
        \coordinate [label=below:$v_4$] (v4) at (.75,-.75);
        \circleat{.75}{-.75}{black};

        \coordinate [label=left:$v_5$] (v5) at (0.75,0.25);
        \circleat{0.75}{0.25}{black};

        \coordinate [label=above:$v_6$] (v6) at (2,1);
        \circleat{2}{1}{black};

        \coordinate [label=above left:$v_7$] (v7) at (0,3);
        \circleat{0}{3}{Dandelion};

        \coordinate [label=above left:$v_8$] (v8) at (-.75,1.5);
        \circleat{-.75}{1.5}{Dandelion};

\coordinate [label=left:$a$] (x) at (-1.5,-1.5);
\coordinate [label=above:$b$] (y) at (3,0);
\coordinate [label=right:$c$] (z) at (0,4);

\draw[Dandelion, very thick] (v3) -- (v8);
\draw[Dandelion, very thick] (v8) -- (v7) -- (v1) -- (v3)--(v8);

\draw[very thick] (v1) -- (v6) -- (v7);

\draw[very thick] (v6) -- (v5);
\draw[very thick] (v4) -- (v3);

\draw[very thick] (v5) --(v4);
\draw[very thick] (v4) -- (v1);

\draw[very thick] (v5) -- (v8);

    \end{tikzpicture}
    \caption{Part 2, Case 2 and Part 3, Case 2 have the same generators, but
different apices; both visible facets are missing.}
    \label{fig: part 2 case 2}
\end{figure}
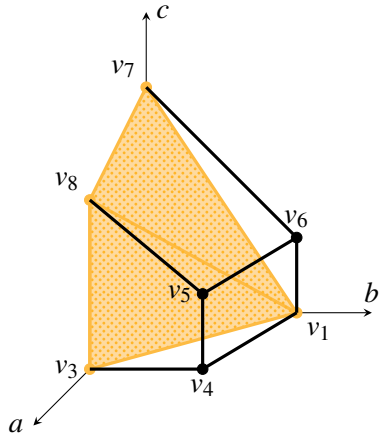
Here, $\Pi(C_{2.2})$ contains four integer points, and the integer-point transform is 
\begin{align*}
    \sigma_{C_{2.2}}(y,z_1,z_2,z_3) = \frac{y^2z_1z_3^3 + y^2z_1z_3^4 +
y^2z_1z_2z_3 + y^2z_1z_2z_3^2}{(1-yz_2^2)(1-yz_1)(1-yz_3^3)(1-yz_1z_3^2)} \, .
    \end{align*}
Imposing the condition that~$b$ is even means disregarding the last two terms in the numerator, from which we obtain the generating function
\begin{align*}
    \genffourk_{2.2}(x,y_1,y_2,y_3,q,t)  &=\frac{x^2y_1y_3^3q^6t^4 +
x^2y_1y_3^4q^5t^5}{(1-xy_2^2q^2t^2)(1-xy_1q^3t^3)(1-xy_3^3q^3t)(1-xy_1y_3^2qt^4)}
\, .
\end{align*}

Part 3, Case 2 proceeds in a similar fashion.
We define the cone~$C_{3.2}$ by the following system of inequalities reduced from inequalities~\eqref{ineq:k^4 1}-\eqref{ineq:k^4 3} and~\eqref{ineq:k^4 part 3.2}:

\begin{align*} 
    0 &\leq b,\\ 
    b &< 2k-2a,\\ 
    c &\geq 3k-3a-\frac{3(b+1)}{2}+1,\\ 
    c &< 3k-a-\frac{3(b+1)}{2}+1.
\end{align*}
The cone~$C_{3.2}$ is a translation of the cone~$C_{2.2}$ such that its apex is at $(0,0,0,-\frac{1}{2})$.
Again, $\Pi(C_{3.2})$ contains four integer points, and the resulting integer-point transform is
\begin{align*}
    \sigma_{C_{3.2}}(y,z_1,z_2,z_3)
    &= \frac{y^2 z_1 z_2 z_3  + y^2 z_1z_2z_3^2+  y^2 z_1 z_3^{3}+ y^2 z_1
z_3^{4}}{(1-yz_2^2)(1-yz_1)(1-yz_3^3)(1-yz_1z_3^2)} \, .
    \end{align*}
To impose the condition that~$b$ is odd, we disregard the last two terms in the numerator, from which we obtain the generating function
\begin{align*}
    \genffourk_{3.2}(q,t, x,y_1,y_2,y_3) &=\frac{ x^2 y_1 y_2 y_3q^5 t^5 (q   +
y_3t)}{(1-xy_2^2q^2t^2)(1-xy_1q^3t^3)(1-q^3txy_3^3)(1-xy_1y_3^2qt^4)} \, .
\end{align*}

\subsection*{Case 3}\label{subsec:part2,3_3}
Let~$C_{2.3}$ be the cone defined by inequalities ~\eqref{ineq:k^4 1}-\eqref{ineq:k^4 3} and \eqref{ineq:k^4 part 2.3}; it has generators~$v_1, v_2, v_3, v_7$, and we discard the facet opposite to~$v_2$. 
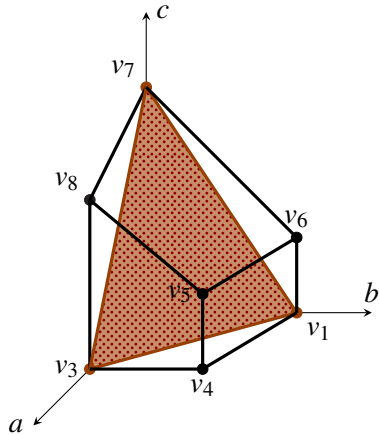
\begin{figure}[htbp!]
    \centering
        \begin{tikzpicture}

\begin{scope}[on background layer]

\draw[-stealth] (v2) to (z);
\draw[very thick] (v2) -- (v7);
\draw[very thick] (v2) -- (v1);
\draw[-stealth] (v2) to (y);
\draw[-stealth] (v2) to (x);
\draw[very thick] (v2) -- (v3);

        \coordinate [label=left:$v_2$] (v2) at (0,0);
        \circleat{0}{0}{black};

\filldraw[RawSienna!50] (v1) --(v3)--(v7)-- (v1);

\filldraw[pattern color=RawSienna, pattern=crosshatch dots] (v1) -- (v3) -- (v7) -- (v1);

\end{scope}

        \coordinate [label=below right:$v_1$] (v1) at (2,0);
        \circleat{2}{0}{RawSienna};
        
        \coordinate [label=left:$v_3$] (v3) at (-.75,-.75);
        \circleat{-.75}{-.75}{RawSienna};
        
        \coordinate [label=below:$v_4$] (v4) at (.75,-.75);
        \circleat{.75}{-.75}{black};

        \coordinate [label=left:$v_5$] (v5) at (0.75,0.25);
        \circleat{0.75}{0.25}{black};

        \coordinate [label=above:$v_6$] (v6) at (2,1);
        \circleat{2}{1}{black};

        \coordinate [label=above left:$v_7$] (v7) at (0,3);
        \circleat{0}{3}{RawSienna};

        \coordinate [label=above left:$v_8$] (v8) at (-.75,1.5);
        \circleat{-.75}{1.5}{Black};

\coordinate [label=left:$a$] (x) at (-1.5,-1.5);
\coordinate [label=above:$b$] (y) at (3,0);
\coordinate [label=right:$c$] (z) at (0,4);

\draw[very thick] (v3) -- (v8);
\draw[very thick] (v8) -- (v7);
\draw[very thick, RawSienna] (v7) -- (v1) -- (v3)--(v7);

\draw[very thick] (v1) -- (v6) -- (v7);

\draw[very thick] (v6) -- (v5);
\draw[very thick] (v4) -- (v3);

\draw[very thick] (v5) --(v4);
\draw[very thick] (v4) -- (v1);

\draw[very thick] (v5) -- (v8);

    \end{tikzpicture}
    \caption{Part 2, Case 3 and Part 3, Case 3 have the same generators but
different apices; the visible face is missing.}
    \label{fig: part 2 case 3}
\end{figure}
Since \[\Pi(C_{2.3}) \cap \Z^4 =
\{(1,0,0,0),(1,0,1,0),(1,0,0,1),(1,0,1,1),(1,0,0,2),(2,0,1,2)\} \, ,\] the
integer-point transform is 
\begin{align*}
    \sigma_{C_{2.3}}(y,z_1,z_2,z_3) &=\frac{y + yz_2 + yz_3 + yz_2z_3 + yz_3^2
+ y^2z_2z_3^2}{(1-yz_2^2)(1-y)(1-yz_1)(1-yz_3^3)} \, . 
    \end{align*}
We are only interested in the points where~$b$ is even, and so we disregard the three terms in the numerator with $z_2$, arriving at the generating function
\begin{align*}
    \genffourk_{2.3}(q,t, x,y_1,y_2,y_3) &= \frac{ xq^4(q^2 + y_3q t +
y_3^2t)}{(1-xy_2^2q^2t^2)(1- xq^6)(1-xy_1q^3t^3)(1-xy_3^3q^3t)} \, .
\end{align*}

For Part 3, Case 3, let~$C_{3.3}$ be the cone defined by inequalities ~\eqref{ineq:k^4 1}-\eqref{ineq:k^4 3} and \eqref{ineq:k^4 part 3.3}; it has the same generators as~$C_{2.3}$, but is translated to place its apex at~$(\frac{1}{6}, 0 ,0,0)$ rather than the origin.
Since \[\Pi(C_{3.3}) \cap \Z^4 = \{(1,0,0,0),(1,0,1,0),(1,0,0,1),(2,0,1,1),(1,0,0,2),(2,0,1,2)\} \, ,\] 
the integer-point transform is
\begin{align*}
    \sigma_{C_{3.3}}(y,z_1,z_2,z_3)   &= \frac{y + yz_2 + yz_3 + y^2z_2z_3 +
yz_3^2 + y^2 z_2z_3^2}{(1-yz_2^2)(1-y)(1-yz_1)(1-yz_3^3)} \, .
    \end{align*}
We are now only interested in the points where~$b$ is odd and so we disregard the three terms in the numerator without $z_2$; the resulting generating function is 
\begin{align*}
    \genffourk_{3.3}(q,t, x,y_1,y_2,y_3) &=  \frac{xy_2q^4t^2 + x^2y_2y_3q^9t^2 +
 x^2 y_2y_3^2q^8t^3}{(1-xy_2^2q^2t^2)(1-xq^6)(1-xy_1q^3t^3)(1-xy_3^3q^3t)} \, .
\end{align*}

Adding the eight generating functions  yields
\begin{align*}
    \genffourk(q,t, x,y_1,y_2,y_3) =&\frac{1-q^2t^8x^2y_1y_2^3y_3}{(1-q^2t^2xy_2^2)(1-qt^5xy_1y_2)(1-qt^3xy_2^2y_3)(1-qt^4xy_1y_3^2)(1-t^6xy_1y_2y_3)}\\
    & {} + \frac{xy_1q^2t^3(q + ty_3)}{(1-q^2t^2xy_2^2)(1-q^3t^3xy_1)(1-qt^5xy_1y_2)(1-qt^4xy_1y_3^2)}\\
& {} + 
\frac{q^3txy_3^3}{(1-q^2t^2xy_2^2)(1-q^3txy_3^3)(1-qt^3xy_2^2y_3)(1-qt^4xy_1y_3^2)}\\
& {} + \frac{q^5t^4x^2y_1y_3^3(q+ty_3)}{(1-q^2t^2xy_2^2)(1-q^3txy_3^3)(1-q^3t^3xy_1)(1-qt^4xy_1y_3^2)}\\
& {} + \frac{ q^4x(q^2 + ty_3^2 + q ty_3)}{(1-q^2t^2xy_2^2)(1- q^6x)(1-q^3txy_3^3)(1-q^3t^3xy_1)}\\
& {} + \frac{q^2 t^2xy_2y_3 ( q + ty_3-q^2t^3xy_2^2y_3 ) }{(1-q^2t^2xy_2^2)(1-q^3txy_3^3)(1-qt^3xy_2^2y_3)(1-qt^4xy_1y_3^2)}\\
& {} + \frac{ x^2 y_1 y_2 y_3q^5 t^5 (q   + y_3t)}{(1-q^2t^2xy_2^2)(1-q^3txy_3^3)(1-q^3t^3xy_1)(1-qt^4xy_1y_3^2)} \\
& {} + \frac{q^4t^2xy_2(1+q^5xy_3+q^4txy_3^2)}{(1-q^2t^2xy_2^2)(1-q^6x)(1-q^3txy_3^3)(1-q^3t^3xy_1)}\, .\\
\end{align*}
Setting $y_2=y_3=y_4=1$ we obtain
\begin{equation*}
  \genffourk(x,1,1,1,q,t)=\frac{N}{(1-q^3tx)(1-qt^3x)(1-q^2t^2x)(1-q^6x)(1-t^6x)} \, ,
\end{equation*}
where 
\begin{align*}
    N=1+& \left(q^5t+qt^5+q^4t^2+q^2t^4+q^4t+qt^4+q^3t^2+q^2t^3+q^3t^3\right)x\\
    & {} + \left(-q^7t^3-q^3t^7+q^6t^5+q^5t^6-q^6t^4-q^4t^6-q^5t^5-q^5t^4-q^4t^5\right)x^2\\
    & {} - \left(q^8t^8+q^9t^6+q^6t^9+q^8t^7+q^7t^8\right)x^3 ,
\end{align*}
and this completes the proof of Theorem~\ref{thm:kkkk}.  


\section{The case \texorpdfstring{$\vec{k} = (k,a,\dots,a)$}{ka}}\label{sec:kaaa}

Based on computational data, we make the following conjecture about $q,t$-symmetry of the refined $q,t$-Catalan numbers $C_{\vec{k}}(q,t)$.

\begin{conjecture}\label{conj:k-symmetry}
If $\vec{k} = (k,a,\ldots,a)$, for positive integers $k,a$, then the refined $q,t$-Catalan numbers $C_{\vec{k}}(q,t)$ are $q,t$-symmetric.
\end{conjecture}

We settle this conjecture in the four-dimensional case when $k \le a$.  As in the case of $\vec{k}=(k,k,k,k)$, our current understanding of the bounce formula makes further progress daunting. 
Our goal in this section is to compute~\eqref{eq:gfctkaaa}, i.e.,
\[
  \genffourkm(x, y, q, t) := \sum_{ k, m \in \Z_{ \ge 0 } } C_{\overrightarrow{k} =
(k,k+m,k+m,k+m)}(q,t) \, x^k \, y^m.
\]

\begin{theorem}\label{thm:kaaa}
    If $k,m\ge 0$, then
\begin{align*}
    \genffourkm(x,y,q,t)=\frac{M}{(1-xq^6)(1-xt^6)(1-xq^3t)(1-xqt^3)(1-xq^2t^2)(1-yq^3)(1-yt^3)(1-yqt)}\, ,
\end{align*}
where
\begin{align*}
    M=
    (q^{13}t^7 &+ q^7t^{13} + q^{12}t^8 + q^8t^{12} + q^9t^{12} + q^{12}t^9 + q^{11}t^{11} + q^{10}t^{11}+ q^{11}t^{10} + q^9t^{11} + q^{11}t^9  + q^{10}t^{10})x^3y^2 \\
    &- (q^{11}t^5 + q^5t^{11} + q^{10}t^6 + q^6t^{10} + q^8t^9 + q^9t^8 + q^7t^9 + q^9t^7 + q^8t^8 - q^7t^8 - q^8t^7 + q^7t^7)x^2y^2 \\
     &-(q^7t^4 + q^4t^7 + q^6t^5 + q^5t^6)xy^2\\
    &-(q^{12}t^6 + q^6t^{12} + q^8t^{11} + q^{11}t^8 + q^{10}t^8 + q^8t^{10} + q^7t^{11} + q^{11}t^7  + q^7t^{10} + q^{10}t^7 + 2q^9t^9 + q^8t^9 + q^9t^8)x^3y\\ 
    & -(q^{10}t^3 + q^3t^{10} + q^5t^7 + q^7t^5 - q^5t^9 -q^9t^5 + q^4t^9 + q^9t^4 + q^4t^8 + q^8t^4 + q^5t^6 + q^6t^5-q^8t^6 - q^6t^8)x^2y  \\
    &+(qt^8+ q^8t + qt^7  + q^7t + q^6t^3 +q^3t^6 + q^4t^5 + q^5t^4  + q^2t^5 + q^5t^2 + 2q^4t^4 + q^5t^3+ q^3t^5 +2q^3t^4 +2q^4t^3 \\
    & \qquad \;\;\; + q^7t^2 + q^2t^7 + q^6t^2 + q^2t^6)xy \\
    &-(q^2t+qt^2)y\\
    &+(q^9t^6 + q^6t^9 + q^8t^7+ q^7t^8 + q^8t^8 )x^3 + (q^7t^3 + q^3t^7 - q^5t^6 - q^6t^5 + q^4t^6 + q^6t^4 + q^5t^5 + q^4t^5 + q^5t^4)x^2  \\
    &-(qt^5 + q^5t + q^2t^4 + q^4t^2 + qt^4 + q^4t + q^3t^3 + q^2t^3 + q^3t^2)x \;-\; 1 \,.
\end{align*}
\end{theorem}
Note once more the $q,t$-symmetry in this formula. 

A $(k,k+m,k+m,k+m)$-Dyck path $D$ can be uniquely identified with the vector $(k,m,a,b,c)$ based on the red ranks $(0,r_2,r_3,r_4)$ where
\[
r_2=k-a, \quad r_3=2k+m-a-b, \quad r_4=3k+2m-a-b-c.
\]
We compute $\text{area}(D)=6k+3m-3a-2b-c$ and
    \[
    \text{bounce(D)}=
\left\{
\begin{array}{ll}
   3a  &  b,c=0,\\
   3a+\left\lceil \frac{c}{3}\right\rceil  & b=0, \, 0<c\le 3(k-a),\\
   2a+k+\left\lceil\frac{c-3(k-a)}{2}\right\rceil & b=0, \, c>3(k-a),\\
   3a+2\left\lceil\frac{b}{2}\right\rceil & 0<b\le 2(k-a),\, c=0,\\
   3a+2\left\lceil\frac{b}{2}\right\rceil+\left\lceil\frac{c}{3}\right\rceil & 0<b\le 2(k-a),\, b \text{ is even and } \\ &\quad \quad 0< c\le 3\left(k-a-\left\lceil\frac{b}{2}\right\rceil\right),\\
   3a+2\left\lceil\frac{b}{2}\right\rceil+\left\lceil\frac{c-1}{3}\right\rceil & 0<b\le 2(k-a),\, b \text{ is odd and } \\
   & \quad \quad 0\le c-1\le 3\left(k-a-\left\lceil\frac{b}{2}\right\rceil\right),\\
   2a+\left\lceil\frac{b}{2}\right\rceil+k+\left\lceil\frac{c-3(k-a-\left\lceil\frac{b}{2}\right\rceil)}{2}\right\rceil & 0<b\le 2(k-a),\, b \text{ is even and } \, \\
   & \quad \quad 3\left(k-a-\left\lceil\frac{b}{2}\right\rceil\right)<c\le 3(k-\left\lceil\frac{b}{2}\right\rceil)-a+2m,\\
   2a+\left\lceil\frac{b}{2}\right\rceil+k+ &  0<b\le 2(k-a),\, b \text{ is odd and } \, \\
   \quad \quad \left\lceil\frac{c-1-3\left(k-a-\left\lceil\frac{b}{2}\right\rceil\right)}{2}\right\rceil& \quad \quad 3\left(k-a-\left\lceil\frac{b}{2}\right\rceil\right)<c-1\le 3(k-\left\lceil\frac{b}{2}\right\rceil)-a+2m,\\
   c-2k+4a+4\left\lceil\frac{b}{2}\right\rceil-m & 0<b\le 2(k-a),\, b \text{ is even and } \\
   & \quad \quad c> 3(k-\left\lceil\frac{b}{2}\right\rceil)-a+2m,\\
   c-1-2k+4a+4\left\lceil\frac{b}{2}\right\rceil-m & 0<b\le 2(k-a),\, b \text{ is odd and } \\
   & \quad \quad c-1> 3(k-\left\lceil\frac{b}{2}\right\rceil)-a+2m,\\
   5a+2b-2k+\left\lceil\frac{c}{2}\right\rceil & b>2(k-a), \, 0\le c\le 2(2k-a+m-b),\\
   6a+3b-4k+c-m & b>2(k-a), \,c>2(2k-a+m-b).
\end{array}\right.
\]
By definition, all vectors $(k,m a, b, c)$ corresponding
to $(k, k+m, k+m, k+m)$-Dyck paths must satisfy the following inequalities:
\begin{align}
    0&\le a\le k \label{ineq: km all satisfy 1},\\
    0&\le m \label{ineq: km all satisfy 2},\\
    0&\le b\le 2k+m-a \label{ineq: km all satisfy 3},\\
    0&\le c\le 3k+2m-a-b. \label{ineq: km all satisfy 4}
\end{align}
This system of inequalities corresponds to the cone $C$ generated by 
\begin{align*}
    \begin{bmatrix}
        1\\0\\0\\2\\0
    \end{bmatrix},
    \begin{bmatrix}
        1\\0\\1\\0\\0
    \end{bmatrix},
    \begin{bmatrix}
        1\\0\\0\\0\\0
    \end{bmatrix},
    \begin{bmatrix}
        0\\1\\0\\0\\0
    \end{bmatrix},
    \begin{bmatrix}
        0\\1\\0\\1\\0
    \end{bmatrix},
    \begin{bmatrix}
        1\\0\\1\\1\\0
    \end{bmatrix},
    \begin{bmatrix}
        1\\0\\1\\1\\1
    \end{bmatrix},
    \begin{bmatrix}
        1\\0\\0\\2\\1
    \end{bmatrix},
    \begin{bmatrix}
        1\\0\\1\\0\\2
    \end{bmatrix},
    \begin{bmatrix}
        1\\0\\0\\0\\3
    \end{bmatrix},
    \begin{bmatrix}
        0\\1\\0\\0\\2
    \end{bmatrix},
    \begin{bmatrix}
        0\\1\\0\\1\\1
    \end{bmatrix}.
\end{align*}
We denote these cone generators by $v_1,\dots,v_{12}$. 
Similarly to the $(k,k,k,k)$ case we triangulate $C$ depending on the bounce cases listed above.  We denote
\begin{align}
    \textbf{Part 1: } & b=0 \label{ineq: part 1},\\
    \textbf{Part 2: } & 0<b\le 2(k-a) \, \text{and} \, c=0 \label{ineq: part 2},\\
    \textbf{Part 3: } & 0<b\le 2(k-a), \, c\ne 0,\, \text{and $b$ is even}\label{ineq: part 3},\\
    \textbf{Part 4: } & 0<b\le 2(k-a),\,  c\ne 0,\, \text{ and $b$ is odd} \label{ineq: part 4},\\
    \textbf{Part 5: } & b>2(k-a). \label{ineq: part 5}
\end{align}
Some parts are further split into subcases.
We do not exhibit the details of the ensuing computations, as the generating-function
methods are similar to those for the $\vec{k}=(k_1,k_2,k_3)$ and $\vec{k}=(k,k,k,k)$ cases above.
However, we do give the result for each subcase.

\subsection*{Part 1}
We divide the points with $b=0$ into three cases: 
\begin{align}
    \textbf{Case 1: }& c=0\label{ineq: km p1 c1},\\
    \textbf{Case 2: }& 0<c\le 3(k-a)\label{ineq: km p1 c2},\\
    \textbf{Case 3: }& c>3(k-a).\label{ineq: km p1 c3}
\end{align}
Case 1 gives rise to the unimodular cone $C_{1.1}$ defined by the inequalities~\ref{ineq: km all satisfy 1}-\ref{ineq: km all satisfy 4},~\ref{ineq: part 1}, and~\ref{ineq: km p1 c1},
with generating function 
\[
\genffourkm_{1.1}(x,y,z_1,q,t)=\frac{1}{(1-xz_1q^3t^3)(1-yq^3)(1-xq^6)}\, .
\]
Case 2 gives the simplicial cone $C_{1.2}$ defined by the inequalities~\ref{ineq: km all satisfy 1}-\ref{ineq: km all satisfy 4},~\ref{ineq: part 1} and~\ref{ineq: km p1 c2}, 
with generating function
\[
\genffourkm_{1.2}(x,y,z_1,z_3,q,t)=\frac{xz_3q^3t(q^2+z_3q+z_3^2)}{(1-xz_3^3q^3t)(1-xq^6)(1-xz_1q^3t^3)(1-yq^3)}\, .
\]
Case 3 gives the cone $C_{1.3}$ defined by the inequalities~\ref{ineq: km all satisfy
1}-\ref{ineq: km all satisfy 4},~\ref{ineq: part 1} and~\ref{ineq: km p1 c3}. It needs to
be triangulated, giving rise to the generating function
\begin{align*}
\genffourkm_{1.3}(x,y,z_1,z_3,q,t)
&= \frac{yz_3q^2t+yz_3^2qt}{(1-xz_3^3q^3t)(1-yq^3)(1-xz_1q^3t^3)(1-yz_3^2qt)} \\
&\qquad + \frac{xz_1z_3q^2t^4+xz_1z_3^2qt^4}{(1-xz_3^3q^3t)(1-xz_1q^3t^3)(1-xz_1z_3^2qt^4)(1-yz_3^2qt)}
\, .
\end{align*}

\subsection*{Part 2}
Let $C_2$ be the simplicial cone defined by the inequalities~\ref{ineq: km all satisfy 1}-\ref{ineq: km all satisfy 4} and~\ref{ineq: part 2}. 
It gives rise to the generating function 
\[
\genffourkm_2(x,y,z_1,z_2,z_3,q,t)=\frac{xz_2q^4t^2 + xz_2^2q^2t^2}{(1-xz_1q^3t^3)(1-yq^3)(1-xz_2^2q^2t^2)(1-xq^6)}\,.
\]

\subsection*{Part 3}
We divide the points with $b$ even, $c\ne 0$, and $0<b\le 2(k-a)$ into three cases:
\begin{align}
    \textbf{Case 1: }& 0<c\le 3\left(k-a-\frac{b}{2}\right)\label{ineq: part 3 case 1}\, ,\\
    \textbf{Case 2: }& 3\left(k-a-\frac{b}{2}\right)<c\le 3\left(k-\frac{b}{2}\right)-a+2m\label{ineq: part 3 case 2}\, ,\\
    \textbf{Case 3: }& c>3\left(k-\frac{b}{2}\right)-a+2m\, .\label{ineq: part 3 case 3}
\end{align}
Case 1 gives the simplicial cone $C_{3.1}$ defined by the inequalities~~\ref{ineq: km all
satisfy 1}-\ref{ineq: km all satisfy 4},~\ref{ineq: part 3} and~\ref{ineq: part 3 case 1};
we need only consider the points with even $b$ values.  
This gives the generating function 
\[
\genffourkm_{3.1}(x,y,z_1,z_2,z_3,q,t)=\frac{x^2z_2^2z_3q^7t^3+x^2z_2^2z_3^2q^6t^3+x^2z_2^2z_3^3q^5t^3}{(1-yq^3)(1-xz_2^2q^2t^2)(1-xz_1q^3t^3)(1-xz_3^3q^3t)(1-xq^6)}\, .
\]
Case 2 gives the cone $C_{3.2}$ defined by inequalities~\ref{ineq: km all satisfy
1}-\ref{ineq: km all satisfy 4},~\ref{ineq: part 3} and~\ref{ineq: part 3 case 2}; again we
need only consider the points with even $b$ value.  
After a triangulation, we compute the generating function
\begin{align*}
\genffourkm_{3.2}(x,y,z_1,z_2,z_3,q,t)
&=
\frac{xyz_2^2z_3q^4t^3+xyz_2^2z_3^2q^3t^3}{(1-yq^3)(1-yz_3^2qt)(1-xz_3^3q^3t)(1-xz_2^2q^2t^2)(1-xz_1q^3t^3)}
\\
&\qquad + \frac{x^2z_1z_2^2z_3q^4t^6+x^2z_1z_2^2z_3^2q^3t^6}{(1-yz_3^2qt)(1-xz_3^3q^3t)(1-xz_2^2q^2t^2)(1-xz_1q^3t^3)(1-xz_1z_3^2qt^4)}\, .
\end{align*}
Case 3 gives the simplicial cone $C_{3.3}$  defined by the inequalities~\ref{ineq: km all satisfy 1}-\ref{ineq: km all satisfy 4} ,~\ref{ineq: part 3} and~\ref{ineq: part 3 case 3}; once
more we need only consider the points with even $b$ value.  
Its generating function is
\[
\genffourkm_{3.3}(x,y,z_1,z_2,z_3,q,t)=\frac{xz_2^2z_3qt^3}{(1-xz_3^3q^3t)(1-yz_3^2qt)(1-xz_1z_3^2qt^4)(1-xz_2^2q^2t^2)(1-xz_2^2z_3qt^3)}\, .
\]

\subsection*{Part 4}
We similarly divide the points with $b$ odd and $0<b\le 2(k-a)$ into three cases:
\begin{align}
 \textbf{Case 1: }  & 0\le c-1\le 3\left(k-a-\frac{b+1}{2}\right) \label{ineq: part 4 case 1}\, ,\\
     \textbf{Case 2: } & 3\left(k-a-\frac{b+1}{2}\right)<c-1\le 3\left(k-\frac{b+1}{2}\right)-a+2m \label{ineq: part 4 case 2}\, ,\\
    \textbf{Case 3: } & 3\left(k-\frac{b+1}{2}\right)-a+2m<c-1\, . \label{ineq: part 4 case 3}
\end{align}
For Case 1, we consider the points with odd $b$ value of the cone $C_{4.1}$ defined by the inequalities~\ref{ineq: km all satisfy 1}-\ref{ineq: km all satisfy 4},~\ref{ineq: part 4} and~\ref{ineq: part 4 case 1}, the last of which in this context simplifies to  
\begin{equation*}
0<c<3k-3a-\frac{3b}{2}\, .
\end{equation*}
It comes with the generating function 
\[
\genffourkm_{4.1}(x,y,z_1,z_2,z_3,q,t)=\frac{xz_2z_3q^3t^2+x^2z_2z_3^2q^8t^3+x^2z_2z_3^3q^7t^3}{(1-yq^3)(1-xz_2^2q^2t^2)(1-xz_1q^3t^3)(1-xz_3^3q^3t)(1-xq^6)}\, .
\]
Similarly, Case 2 is captured by the points with odd $b$ value of the cone $C_{4.2}$ defined by the inequalities~\ref{ineq: km all satisfy 1}-\ref{ineq: km all satisfy 4},~\ref{ineq: part 4}, and~\ref{ineq: part 4 case 2}, the last of which in this context simplifies to
\[
3k-3a-\frac{3b}{2}\le c<3k-\frac{3b}{2}-a+2m\, .
\]  
After a triangulation, its generating function is
\begin{align*}
  \genffourkm_{4.2}(x,y,z_1,z_2,z_3,q,t)
  &=
\frac{xyz_2z_3^2q^5t^3+xyz_2z_3^3q^4t^3}{(1-xz_3^3q^3t)(1-xz_2^2q^2t^2)(1-yz_3^2qt)(1-xz_1q^3t^3)(1-yq^3)}\\
&\qquad {} + \frac{x^2z_1z_2z_3^2q^5t^6+x^2z_1z_2z_3^3q^4t^6}{(1-xz_3^3q^3t)(1-xz_2^2q^2t^2)(1-yz_3^2qt)(1-xz_1q^3t^3)(1-xz_1z_3^2qt^4)}\, .
\end{align*}
For Case 3, we consider the points with odd $b$ value of the cone $C_{4.3}$ defined by the inequalities~\ref{ineq: km all satisfy 1}-\ref{ineq: km all satisfy 4},~\ref{ineq: part 4}, and~\ref{ineq: part 4 case 3}, the last of which simplifies to
\[
c\ge 3k-\frac{3b}{2}-a+2m\, .
\]
The corresponding generating function is 
\[
\genffourkm_{4.3}(x,y,z_1,z_2,z_3,q,t)=\frac{xz_2z_3^2q^2t^3}{(1-xz_3^3q^3t)(1-yz_3^2qt)(1-xz_1z_3^2qt^4)(1-xz_2^2q^2t^2)(1-xz_2^2z_3qt^3)}\, .
\]

\subsection*{Part 5}
We divide the points with $b>2(k-a)$ into two cases: 
\begin{align}
    \textbf{Case 1: } & 0\le c\le 2(2k-a+m-b)\, ,\label{ineq: km p5 c1} \\
    \textbf{Case 2: } & c>2(2k-a+m-b)\, . \label{ineq: km p5 c2}
\end{align}
For Case 1, we consider the cone $C_{5.1}$ defined by the inequalities~\ref{ineq: km all satisfy 1}-\ref{ineq: km all satisfy 4},~\ref{ineq: part 5}, and~\ref{ineq: km p5 c1}. 
After a triangulation, we compute its generating function as
\begin{align*}
  \genffourkm_{5.1}(x,y,z_1,z_2,z_3,q,t)
  &=\frac{yz_2qt^2 +
y^2z_2z_3q^3t^3}{(1-yq^3)(1-yz_3^2qt)(1-yz_2qt^2)(1-xz_2^2q^2t^2)(1-xz_1q^3t^3)} \\
  &\qquad {} + \frac{xyz_1z_2z_3q^3t^6+xyz_1z_2z_3^2q^2t^6}{(1-yz_3^2qt)(1-yz_2qt^2)(1-xz_2^2q^2t^2)(1-xz_1q^3t^3)(1-xz_1z_3^2qt^4)}\\
  &\qquad {} + \frac{xz_1z_2 qt^5 + x^2z_1^2z_2z_3q^3t^9}{(1-yz_2qt^2)(1-xz_2^2q^2t^2)(1-xz_1q^3t^3)(1-xz_1z_3^2qt^4)(1-xz_1z_2qt^5)}\, .
\end{align*}
Finally, for Case 2, let $C_{5.2}$ be the cone defined by the inequalities~\ref{ineq: km all satisfy 1}-\ref{ineq: km all satisfy 4},~\ref{ineq: part 5} and~\ref{ineq: km p5 c2}. 
With another triangulation, we compute the generating function 
\begin{align*}
    \genffourkm_{5.2}(x,y,z_1,z_2,z_3,q,t)
    &=\frac{yz_2z_3t^3}{(1-yz_3^2qt)(1-yz_2qt^2)(1-yz_2z_3t^3)(1-xz_2^2q^2t^2)(1-xz_1z_3^2qt^4)}\\
    &\qquad {} + \frac{xyz_2^3z_3^2qt^6}{(1-yz_3^2qt)(1-yz_2z_3t^3)(1-xz_2^2q^2t^2)(1-xz_2^2z_3qt^3)(1-xz_1z_3^2qt^4)}\\
    &\qquad {} + \frac{xyz_1z_2^2z_3qt^8}{(1-yz_2qt^2)(1-yz_2z_3t^3)(1-xz_2^2q^2t^2)(1-xz_1z_3^2qt^4)(1-xz_1z_2qt^5)}\\
    &\qquad {} + \frac{x^2z_1z_2^3z_3q^2t^8}{(1-yz_2z_3t^3)(1-xz_2^2q^2t^2)(1-xz_2^2z_3qt^3)(1-xz_1z_3^2qt^4)(1-xz_1z_2qt^5)}\\
    &\qquad {} + \frac{xz_1z_2z_3t^6}{(1-yz_2z_3t^3)(1-xz_2^2z_3qt^3)(1-xz_1z_3^2qt^4)(1-xz_1z_2qt^5)(1-xz_1z_2z_3t^6)}\, .
\end{align*}

Setting $z_1,z_2,z_3=1$ and adding all the generating functions yields Theorem~\ref{thm:kaaa}.


\section{Further directions}\label{sec:further}

\subsection{Revisiting Xin and Zhang's conjecture}
As previously mentioned, Xin and Zhang~\cite{xin2022qtsymmetry} conjectured that if~$\lambda=((a+1)^s,a^{n-s})$ with~$0\leq s\leq n$, then $C_\lambda(q,t)$ is $q,t$-symmetric. 
Since $C_\lambda(q,t)$ involves taking a sum over possibly multiple $C_{\vec{k}}(q,t)$, it is interesting to further ask for which $\vec{k}$ we have $q,t$-Catalan symmetry. 
For example, Xin and Zhang note that $C_{\lambda = (1,1,1,3)}(q,t)$ is not $q,t$-symmetric. 
Breaking this up into the constituent $C_{\vec{k}}(q,t)$, we find that 
\[ C_{\vec{k} = (1,1,1,3)}(q,t) \qquad \text{ and } \qquad C_{\vec{k} =
(3,1,1,1)}(q,t) \] 
are $q,t$-symmetric (see Figure~\ref{fig:(1,1,1,3)and(3,1,1,1)}), but 
\[ C_{\vec{k} = (1,1,3,1)}(q,t) \qquad \text{ and } \qquad C_{\vec{k} =
(1,3,1,1)}(q,t) \] 
are not $q,t$-symmetric (see Figure~\ref{fig:(1,1,3,1)and(1,3,1,1)}).
Since the asymmetries present in $C_{\vec{k} = (1,1,3,1)}(q,t)$ and $C_{\vec{k} = (1,3,1,1)}(q,t)$ do not cancel out, $C_{\lambda = (1,1,1,3)}(q,t)$ is not $q,t$-symmetric. 

\begin{figure}[ht]
    \centering
    \small
\createTable{6}{1/6/0, 1/5/1, 1/4/2, 1/3/3, 1/2/4, 1/1/5, 1/0/6, 1/4/1, 1/3/2, 1/2/3, 1/1/4, 1/3/1, 1/2/2, 1/1/3, $q^0$/0/-1, $q^1$/1/-1, $q^2$/2/-1, $q^3$/3/-1, $q^4$/4/-1, $q^5$/5/-1, $q^6$/6/-1, $t^0$/-1/0, $t^1$/-1/1, $t^2$/-1/2, $t^3$/-1/3, $t^4$/-1/4, $t^5$/-1/5, $t^6$/-1/6}
\createTable{12}{1/{12}/0, 1/{11}/1, 1/{10}/2, 1/9/3, 1/8/4, 1/7/5, 1/6/6, 1/5/7, 1/4/8, 1/3/9, 1/2/{10}, 1/1/{11}, 1/0/{12}, 1/{10}/1, 1/9/2, 1/8/3, 1/7/4, 1/6/5, 1/5/6, 1/4/7, 1/3/8, 1/2/9, 1/1/{10}, 1/9/1, 2/8/2, 3/7/3, 3/6/4, 3/5/5, 3/4/6, 3/3/7, 2/2/8, 1/1/9, 1/6/3, 1/5/4, 1/4/5, 1/3/6, $q^0$/0/-1, $q^1$/1/-1, $q^2$/2/-1, $q^3$/3/-1, $q^4$/4/-1, $q^5$/5/-1, $q^6$/6/-1, $q^7$/7/-1, $q^8$/8/-1, $q^9$/9/-1, $q^{10}$/10/-1, $q^{11}$/11/-1, $q^{12}$/12/-1, $t^0$/-1/0, $t^1$/-1/1, $t^2$/-1/2, $t^3$/-1/3, $t^4$/-1/4, $t^5$/-1/5, $t^6$/-1/6, $t^7$/-1/7, $t^8$/-1/8, $t^9$/-1/9, $t^{10}$/-1/10, $t^{11}$/-1/11, $t^{12}$/-1/12}
    \caption{Coefficients of $C_{\vec{k} = (1,1,1,3)}(q,t)$ (left) and of $C_{\vec{k} = (3,1,1,1)}(q,t)$ (right). Both are $q,t$-symmetric. Reading the tables: the lower left box is $(0,0)$, the coefficient in box $(i,j)$ corresponds to the summand $q^i t^j$. }
    \label{fig:(1,1,1,3)and(3,1,1,1)}
\end{figure}

\begin{figure}[ht]
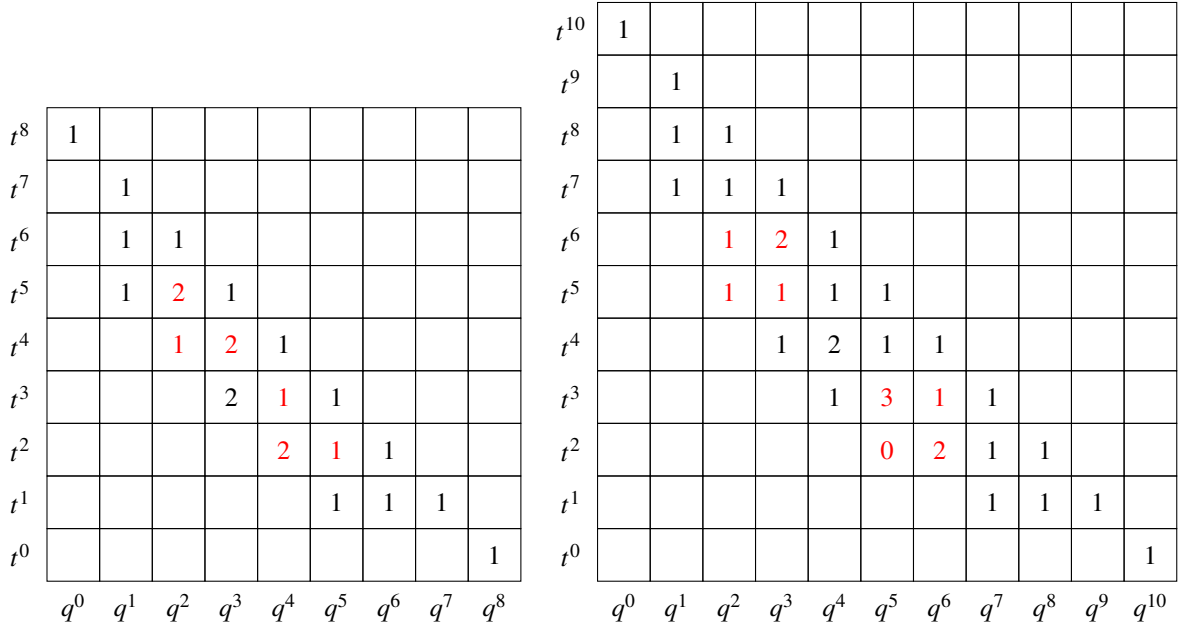

    \centering
    \small
\createTable{8}{1/8/0, 1/7/1, 1/6/2, 1/5/3, 1/4/4, 1/3/5, 1/2/6, 1/1/7, 1/0/8, 1/6/1, \textcolor{red}{1}/5/2, \textcolor{red}{1}/4/3, \textcolor{red}{2}/3/4, \textcolor{red}{2}/2/5, 1/1/6, 1/5/1, \textcolor{red}{2}/4/2, 2/3/3, \textcolor{red}{1}/2/4, 1/1/5, $q^0$/0/-1, $q^1$/1/-1, $q^2$/2/-1, $q^3$/3/-1, $q^4$/4/-1, $q^5$/5/-1, $q^6$/6/-1, $q^7$/7/-1, $q^8$/8/-1, $t^0$/-1/0, $t^1$/-1/1, $t^2$/-1/2, $t^3$/-1/3, $t^4$/-1/4, $t^5$/-1/5, $t^6$/-1/6, $t^7$/-1/7, $t^8$/-1/8}
\createTable{10}{1/{10}/0, 1/9/1, 1/8/2, 1/7/3, 1/6/4, 1/5/5, 1/4/6, 1/3/7, 1/2/8, 1/1/9, 1/0/{10}, 1/8/1, 1/7/2, \textcolor{red}{1}/6/3, 1/5/4, 1/4/5, \textcolor{red}{2}/3/6, 1/2/7, 1/1/8, 1/7/1, \textcolor{red}{2}/6/2, \textcolor{red}{3}/5/3, 2/4/4, \textcolor{red}{1}/3/5, \textcolor{red}{1}/2/6, 1/1/7, 1/4/3, 1/3/4, \textcolor{red}{1}/2/5, \textcolor{red}{0}/5/2, $q^0$/0/-1, $q^1$/1/-1, $q^2$/2/-1, $q^3$/3/-1, $q^4$/4/-1, $q^5$/5/-1, $q^6$/6/-1, $q^7$/7/-1, $q^8$/8/-1, $q^9$/9/-1, $q^{10}$/10/-1, $t^0$/-1/0, $t^1$/-1/1, $t^2$/-1/2, $t^3$/-1/3, $t^4$/-1/4, $t^5$/-1/5, $t^6$/-1/6, $t^7$/-1/7, $t^8$/-1/8, $t^9$/-1/9, $t^{10}$/-1/10}
    \caption{Coefficients of $C_{\vec{k} = (1,1,3,1)}(q,t)$ (left) and of $C_{\vec{k} = (1,3,1,1)}(q,t)$ (right).
    Both are not $q,t$-symmetric, with asymmetries highlighted in red.}
    \label{fig:(1,1,3,1)and(1,3,1,1)}
\end{figure}

This contrasts with examples which fall within Xin and Zhang's conjecture, such as $C_{\lambda = (1,1,1,2)}(q,t)$ which is $q,t$-symmetric. 
There, 
\[ C_{\vec{k} = (1,1,1,2)}(q,t) \qquad \text{ and } \qquad C_{\vec{k} =
(2,1,1,1)}(q,t) \] 
are $q,t$-symmetric (see Figure~\ref{fig:(1,1,1,2)and(2,1,1,1)}), but 
\[ C_{\vec{k} = (1,1,2,1)}(q,t) \qquad \text{ and } \qquad C_{\vec{k} =
(1,2,1,1)}(q,t) \] 
are not $q,t$-symmetric (see Figure~\ref{fig:(1,1,2,1)and(1,2,1,1)}).
However, this time the asymmetries present in $C_{\vec{k} = (1,1,2,1)}(q,t)$ and $C_{\vec{k} = (1,2,1,1)}(q,t)$ do cancel out, so $C_{\lambda = (1,1,1,2)}(q,t)$ is $q,t$-symmetric.

\begin{figure}[ht]
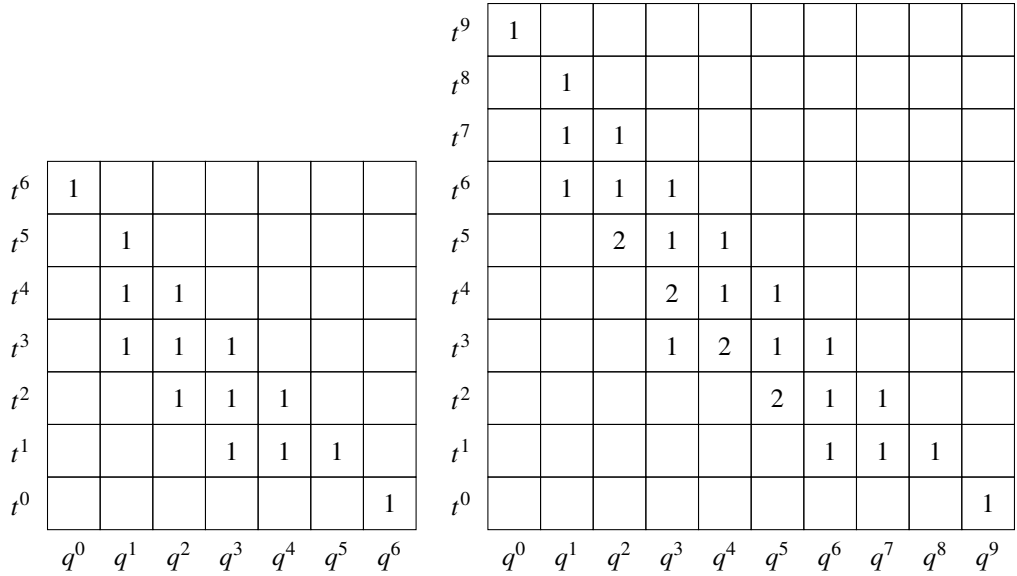

    \centering
    \small
\createTable{6}{1/6/0, 1/5/1, 1/4/2, 1/3/3, 1/2/4, 1/1/5, 1/0/6, 1/4/1, 1/3/2, 1/2/3, 1/1/4, 1/3/1, 1/2/2, 1/1/3, $q^0$/0/-1, $q^1$/1/-1, $q^2$/2/-1, $q^3$/3/-1, $q^4$/4/-1, $q^5$/5/-1, $q^6$/6/-1, $t^0$/-1/0, $t^1$/-1/1, $t^2$/-1/2, $t^3$/-1/3, $t^4$/-1/4, $t^5$/-1/5, $t^6$/-1/6}
\createTable{9}{1/9/0, 1/8/1, 1/7/2, 1/6/3, 1/5/4, 1/4/5, 1/3/6, 1/2/7, 1/1/8, 1/0/9, 1/7/1, 1/6/2, 1/5/3, 1/4/4, 1/3/5, 1/2/6, 1/1/7, 1/6/1, 2/5/2, 2/4/3, 2/3/4, 2/2/5, 1/1/6, 1/3/3, $q^0$/0/-1, $q^1$/1/-1, $q^2$/2/-1, $q^3$/3/-1, $q^4$/4/-1, $q^5$/5/-1, $q^6$/6/-1, $q^7$/7/-1, $q^8$/8/-1, $q^9$/9/-1, $t^0$/-1/0, $t^1$/-1/1, $t^2$/-1/2, $t^3$/-1/3, $t^4$/-1/4, $t^5$/-1/5, $t^6$/-1/6, $t^7$/-1/7, $t^8$/-1/8, $t^9$/-1/9}
    \caption{Coefficients of $C_{\vec{k} = (1,1,1,2)}(q,t)$ (left) and of $C_{\vec{k} = (2,1,1,1)}(q,t)$ (right).
    Both are $q,t$-symmetric. 
    Reading the tables: the lower left box is $(0,0)$, the coefficient in box $(i,j)$ corresponds to the summand $q^i t^j$. }
    \label{fig:(1,1,1,2)and(2,1,1,1)}
\end{figure}

\begin{figure}[ht]
    \centering
        \small
\createTable{7}{1/7/0, 1/6/1, 1/5/2, 1/4/3, 1/3/4, 1/2/5, 1/1/6, 1/0/7, 1/5/1, \textcolor{red}{1}/4/2, 1/3/3, \textcolor{red}{2}/2/4, 1/1/5, 1/4/1, \textcolor{red}{2}/3/2, \textcolor{red}{1}/2/3, 1/1/4, $q^0$/0/-1, $q^1$/1/-1, $q^2$/2/-1, $q^3$/3/-1, $q^4$/4/-1, $q^5$/5/-1, $q^6$/6/-1, $q^7$/7/-1, $t^0$/-1/0, $t^1$/-1/1, $t^2$/-1/2, $t^3$/-1/3, $t^4$/-1/4, $t^5$/-1/5, $t^6$/-1/6, $t^7$/-1/7}
\createTable{8}{1/8/0, 1/7/1, 1/6/2, 1/5/3, 1/4/4, 1/3/5, 1/2/6, 1/1/7, 1/0/8, 1/6/1, 1/5/2, 1/4/3, 1/3/4, 1/2/5, 1/1/6, 1/5/1, \textcolor{red}{2}/4/2, 2/3/3, \textcolor{red}{1}/2/4, 1/1/5, \textcolor{red}{1}/2/3, \textcolor{red}{0}/3/2, $q^0$/0/-1, $q^1$/1/-1, $q^2$/2/-1, $q^3$/3/-1, $q^4$/4/-1, $q^5$/5/-1, $q^6$/6/-1, $q^7$/7/-1, $q^8$/8/-1, $t^0$/-1/0, $t^1$/-1/1, $t^2$/-1/2, $t^3$/-1/3, $t^4$/-1/4, $t^5$/-1/5, $t^6$/-1/6, $t^7$/-1/7, $t^8$/-1/8}
    \caption{Coefficients of $C_{\vec{k} = (1,1,2,1)}(q,t)$ (left) and of $C_{\vec{k} = (1,2,1,1)}(q,t)$ (right). 
    Both are individually not $q,t$-symmetric, with asymmetries highlighted in red.
    Upon combining them, their sum is $q,t$-symmetric.}
    \label{fig:(1,1,2,1)and(1,2,1,1)}
\end{figure}

\subsection{Varying the last parameter}\label{sec:othercases}
We start with a result which shows that the final entry in the vector $\vec{k}$ does not affect its area or bounce.

\begin{proposition}\label{prop: last up step}
    Let 
    \[
    (k_1,\dots,k_j)\in\Z_{>0}^j \quad 
    \text{ and } \quad (a_1,\dots,a_j)\in\Z_{\ge 0}^j
    \]
    be such that the $a_i$ are the number of sequential east steps of a valid $(k_1,\dots,k_j)$-Dyck path (i.e. take $a_1$ east steps after the size $k_1$ north steps, etc.). 
    Let $l,m\in\Z_{>0}$. 
    The bounce and area formulas for the $(k_1,\dots,k_j,l)$-Dyck path and $(k_1,\dots,k_j,m)$-Dyck path with east steps $(a_1,\dots,a_j)$ are identical.
\end{proposition}

\begin{proof}
In the course of applying the bounce algorithm~(\ref{alg: bounce}), bounce can be computed as soon as the top row of the rank tableau is complete.
We claim that all the numbers in the top row of the rank tableau are entered without knowledge of the length of the final north step. 

The initial numbers in the top row are entered while the bounce path lies below the final north steps. Eventually we arrive at a point $P_i$ such that traveling north from $P_i$ the next east steps are along the top boundary of the $n\times n$
square containing the Dyck path.
We count $v_i$, the number of  north steps this last leg of the bounce path traverses, and enter this many copies of $i$ in the rank tableau---completing the top row, and arriving at the moment we can compute bounce.  
Therefore, the values of $l$ and $m$ are immaterial to the bounce computation, only the fact that they are non-zero matters.

For either Dyck path we have
 \begin{align*}
     r_1&=k_1-a_1\, , \\
     r_2&=k_1-a_1+k_2-a_2\, . \\
     &\vdots   \\
     r_j&=k_1-a_1+\dots+k_j-a_j\, ,
 \end{align*}
and since area for either path is the sum of these $r_i$, area also does not depend on $l,m$.
\end{proof}

By combining Proposition~\ref{prop: last up step} with the definition of
$C_{\vec{k}}(q,t)$, we obtain the following result. 

\begin{corollary}\label{cor:last_coordinate}
For any positive integers $k_1,\ldots,k_j,m,l$, 
    \[ C_{\vec{k} = (k_1,\ldots,k_j,m)}(q,t) = C_{\vec{k} =
(k_1,\ldots,k_j,l)}(q,t) \, .\]
\end{corollary}
\begin{proof}
Without loss of generality, assume $m<l$. 
Given a $(k_1,\dots,k_j,m)$-Dyck path, we can add $l-m$ units to the final north step, and $l-m$ east steps after the final east step.  
The result is a valid $(k_1,\dots,k_j,l)$-Dyck path.  
    
On the other hand, given a $(k_1,\dots,k_j,l)$-Dyck path, we can subtract $l-m$ units from the final north step. 
We need to delete $l-m$ east steps from the end in order to create a valid $(k_1,\dots,k_j,m)$-Dyck path. 
The total number of east steps that can occur in the original $(k_1,\dots,k_j,l)$-Dyck path prior to the final north step is $k_1+\dots+k_j$, and the final total is $k_1+\dots+k_j+l$, so at least $l$ east steps must occur after the final north step.  
Therefore, we can delete $l-m$ east steps from the end and still have a valid Dyck path.

We have described a bijection on $(k_1,\dots,k_j,m)$- and $(k_1,\dots,k_j,l)$-Dyck paths, and by Proposition~\ref{prop: last up step} this bijection preserves area and bounce, proving the claim.
\end{proof}

\begin{remark}
We can see Corollary~\ref{cor:last_coordinate} in action in the following examples:
\begin{itemize}
    \item For $m\in\Z_{>0}$ the bounce and area formulas for $(k,k,k,m)$-Dyck paths are the same as the bounce and area formulas for $(k,k,k,k)$-Dyck paths. 
    \item The formula for all $(k_1,k_2,m)$-Dyck paths are also the same regardless of $m\in\Z_{>0}$, the formula for bounce in the $(k_1,k_2,k_3)$-Dyck path case does not depend on $k_3$.
    \item Similarly, Niu points out in \cite{Niu} that the formula for $(k,k,k,k-1)$-Dyck paths is the same as that of $(k,k,k,k)$-Dyck paths.  
    \item If Conjecture~\ref{conj:k-symmetry} holds, Corollary~\ref{cor:last_coordinate} implies that for positive integers $k,a,l$, that $C_{\vec{k}}(q,t)$ with $\vec{k} = (k,a,\ldots,a,l)$ is also $q,t$-symmetric. 
\end{itemize}
Note Corollary~\ref{cor:last_coordinate} does not allow us to extend the bounce formula for a $(k_1,\ldots,k_m)$-Dyck path to a bounce formula for a $(k_1,\ldots,k_m,k_{m+1})$-Dyck path; the fact that we have an additional north step of size $k_{m+1}$ significantly changes the computation. 
\end{remark}

\subsection{Other generalized \texorpdfstring{$q,t$}{q,t}-Catalan numbers}\label{sec:terminology}

We briefly discuss the relationship between Xin-Zhang's refined $q,t$-Catalan numbers and other generalizations which have appeared in the literature. 
We already saw that with $\vec{k} = (\underbrace{1,\ldots,1}_n)$, we recover the classical $q,t$-Catalan numbers $C_n(q,t)$. 
Taking $\vec{k} = (\underbrace{m,\ldots,m}_n)$, we recover the higher $q,t$-Catalan numbers $C_n^{(m)}(q,t)$ \cite{GarsiaHaiman}, which are a special case of the rational $q,t$-Catalan numbers. 
Combinatorial formulas for $C_n^{(m)}(q,t)$ were conjectured by Haglund et al. \cite{HHLRU} and Loehr \cite{Loehr}, and proven as a consequence of the $(m,n)$-shuffle theorem, conjectured by Haglund et al. \cite{HHLRU}, refined by Bergeron et al. \cite{BGLX}, and proven by Mellit \cite{Mellit}.
Note that in both of these cases, there is no difference between $C_{\vec{k}}(q,t)$ and $C_{\lambda}(q,t)$ since there is only one unique rearrangement of the vector. 

In 2020, Gorsky et al. \cite{GorskyHawkesSchillingRainbolt} introduced a
generalized $q,t$-Catalan number, which was further studied by Blasiak et al.
\cite{BHMPS}, where it is defined as
\[C_\textbf{b} := \langle s_{(|\textbf{b}|)}, \omega(D_\textbf{b} \cdot
1)\rangle \, ,\]
$s_{(|\textbf{b}|)}$ is a single-row Schur function, $\omega$ is an involution on symmetric functions, and $D_\mathbf{b}$ is a certain operator in the Schiffmann algebra introduced by Negut. 
It is interesting to speculate on the relationship between $C_\textbf{b}$ and the $q,t$-Catalan numbers of Xin-Zhang, since they depend on a vector of integers, or positive integers, respectively, and generalize the higher $q,t$-Catalan numbers.
For example, the generalized $q,t$-Catalan number 
\[ C_{\mathbf{b} = (1,2,1,0)} = q^4 + q^3t + q^2t^2 + qt^3 + t^4 + q^2t + qt^2
\] 
equals the refined $q,t$-Catalan number 
\[ C_{\vec{k} = (1,2,1)} = q^4 + q^3t + q^2t^2 + qt^3 + t^4 + q^2t + qt^2 \] 
(and neither is a higher $q,t$-Catalan number).
Furthermore, Gorsky et al.\ showed that the generalized $q,t$-Catalan number does not depend on $b_1$, the first entry of the vector $\mathbf{b}$.
This is similar to our remark in Section~\ref{sec:othercases} that the refined $q,t$-Catalan number does not depend on the final entry of $\vec{k}$ (the way the diagrams are drawn is flipped between these settings).

This might suggest that $C_{\mathbf{b}}(q,t)$ and $C_{\vec{k}}(q,t)$ are equal when $\mathbf{b}  = (b_1, \ldots, b_\ell, 0) = (b_\ell, \ldots, b_1) = \vec{k}$, however, the following example shows that, in general, this is not the case: 
\[ C_{\mathbf{b} = (2,1,2,2,0)}(q,t) - C_{\vec{k} = (2,2,1,2)}(q,t) = q^2t^7 +
q^4t^4 + q^3t^5 - q^4t^3 - q^3t^4 - q^2t^5 \, . \]
Perhaps understanding when $C_{\mathbf{b}}(q,t) = C_{\vec{k}}(q,t)$ would allow for some progress on a conjecture of Negut \cite{GorskyHawkesSchillingRainbolt} (and extended by Blasiak et al. \cite[Conjecture 7.1.1]{BHMPS}).

Gorsky et al. \cite[Remark 2.4]{GorskyHawkesSchillingRainbolt} give an example where the coefficients of the generalized $q,t$-Catalan number $C_\textbf{b}$ are negative.
Some of Xin-Zhang's $q,t$-Catalan numbers are not $q,t$-symmetric, while all of the generalized $q,t$-Catalan numbers are $q,t$-symmetric by construction from Negut's operator (see \cite[Equation~(3)]{BHMPS}). 
As a result, neither generalization is contained within the other. 


\section*{Acknowledgements \& Funding}
This work was initiated at the 2023 GEMS of Combinatorics Workshop held at the American Institute of Mathematics (AIM). 
The authors thank the workshop organizers and AIM for fostering a fruitful research atmosphere. 
We thank Carolina Benedetti, Laura Colmenarejo, and Sophie Rehberg for conversations at the start of the project.
We also thank Mark Haiman for helpful conversations.
Lentfer is supported by the National Science Foundation Graduate Research Fellowship DGE-2146752. 
Vindas-Mel\'{e}ndez was partially supported by the National Science Foundation under Award DMS-2102921. 
Waddle is partially supported by the National Science Foundation RTG grant DMS-1840234.

\bibliographystyle{amsplain}
\bibliography{references}


\end{document}